\documentclass[11pt,a4paper,reqno]{amsart}
\usepackage{amsmath,amssymb,a4wide,tikz,color}
\usepackage{xcolor}
\usetikzlibrary{positioning,chains,fit,shapes,calc}

\usepackage{youngtab}

\newcommand{\gas}[2]{(#1,#2)}	
\newcommand{\Rec}[2]{\mathsf{Rec}_{#1}\left(#2\right)}
\newcommand{\MinRec}[2]{\mathsf{Rec}^{\mathsf{min}}_{#1}\left(#2\right)}
\newcommand{\Config}[1]{\mathsf{Configs}\left(#1\right)}

\newcommand{\ctp}{\phi_{CP}}
\newcommand{\ttc}{\phi_{TC}}
\newcommand{\ctt}{\phi_{CT}}
\newcommand{\CanonTop}{\mathsf{CanonTop}}
\newcommand{\red}[1]{\textcolor{red}{#1}}
\newcommand{\ra}[1]{\xrightarrow{#1}}
\newcommand{\T}{\mathcal{T}}
\newcommand{\level}[1]{\mathsf{level}\left(#1\right)}
\newcommand\Levelpoly[3]{\mathrm{Level}_{#1,#2}\left(#3\right)}
\newcommand{\MIN}{\mathsf{LocMin}}
\newcommand{\MAX}{\mathsf{LocMax}}

\newcommand{\Ext}{\mathsf{ext}}
\newcommand{\Int}{\mathsf{int}}

\newcommand\cls{{\mathcal{S}}}
\newcommand\clsn{\cls_n}

\newcommand\tut{\mathsf{T}}

\newcommand\mpi{\mathcal{M}_\pi}
\newcommand\tpi{T_\pi}

\newcommand\nab{{\sc\small NAB}}
\newcommand\nabs{{\sc\small NAB}s}
\newcommand\cnab{{\sc\small CNAB}}
\newcommand\cnabs{{\sc\small CNAB}s}
\newcommand\cmnab{{\sc\small CMNAB}}
\newcommand\cmnabs{{\sc\small CMNAB}s}

\newtheorem{theorem}{Theorem}[section]
 \newtheorem{corollary}[theorem]{Corollary}
 \newtheorem{prop}[theorem]{Proposition}
 
 \newtheorem{lemma}[theorem]{Lemma}
 \newtheorem{fact}[theorem]{Fact}

 \theoremstyle{definition}
 \newtheorem{example}[theorem]{Example}
 
 \newtheorem{problem}[theorem]{Problem}
 \newtheorem{definition}[theorem]{Definition}
 
 \newtheorem{remark}[theorem]{Remark}

 \theoremstyle{remark}
 \newtheorem{algorithm}[theorem]{Algorithm}

\title[Permutation graphs and the Abelian sandpile model]
{Permutation graphs and the Abelian sandpile model, tiered trees and non-ambiguous binary trees}

\author[M. Dukes]{Mark Dukes}
\address{UCD School of Mathematics and Statistics, University College Dublin, Dublin 4, Ireland} 
\email{mark.dukes@ccc.oxon.org}

\author[T. Selig]{Thomas Selig}
\address{Mathematics Division, Science Institute, University of Iceland, Dunhaga 5, 107 Reykjav\'ik, Iceland.}
\email{selig@hi.is}

\author[J.P. Smith]{Jason P. Smith}
\address{Department of Mathematics, University of Aberdeen, Aberdeen AB24 3FX, UK} 
\email{jason.smith@abdn.ac.uk}

\author[E. Steingr\'imsson]{Einar Steingr\'imsson}
\address{Department of Computer and Information Sciences, University of Strathclyde, Glasgow G1 1XH, U.K.} 
\email{einar@alum.mit.edu}

\thanks{Dukes, Selig and Steingr\'imsson were supported by grant EP/M015874/1 from The Engineering and Physical Sciences Research Council.}
\thanks{Smith was supported by grant EP/M027147/1 from The Engineering and Physical Sciences Research Council.}

\date{\today}

\begin{document}
\begin{abstract}
A permutation graph is a graph whose edges are given by inversions of a permutation. We study the Abelian sandpile model (ASM) on such graphs.  We exhibit a bijection between recurrent configurations of the ASM on permutation graphs and the tiered trees introduced by Dugan et al.~\cite{DGGS}. This bijection allows certain parameters of the recurrent configurations to be read on the corresponding tree. In particular, we show that the level of a recurrent configuration can be interpreted as the external activity of the corresponding tree, so that the bijection exhibited provides a new proof of a famous result linking the level polynomial of the ASM to the ubiquitous Tutte polynomial. We show that the set of minimal recurrent configurations is in bijection with the set of complete non-ambiguous binary trees introduced by Aval et al.~\cite{ABBS}, and introduce a multi-rooted generalization of these that we show to correspond to all recurrent configurations. In the case of permutations with a single descent, we recover some results from the case of Ferrers graphs presented in \cite{DSSS}, while we also recover results of Perkinson et al.~\cite{PYY} in the case of threshold graphs.
\end{abstract}

\maketitle
\thispagestyle{empty}

\section{Introduction}\label{sec:intro}

In the Abelian sandpile model (ASM) on a graph, each vertex has a number of ``grains''. If a vertex has at least as many grains as its degree is then it can be toppled, donating one grain to each of its neighbors. If a (nonempty) sequence of topplings from a configuration $c$ of grains leads to $c$ again, then $c$ is said to be recurrent.

In this paper we study the ASM on permutation graphs.  For a permutation $\pi=\pi_1\pi_2\ldots \pi_n$ this is the graph whose vertices are the integers $1,2,\ldots,n$ with an edge between $i$ and $j$ if and only if $i<j$ and $\pi_i>\pi_j$, that is, if $\pi_i$ and $\pi_j$ form an inversion in $\pi$.

This paper generalizes the results in \cite{DSSS}, where the recurrent configurations on Ferrers graphs were classified in terms of decorated EW-tableaux, since Ferrers graphs are isomorphic to permutation graphs of permutations with a single descent 
We extend the bijection in \cite{DSSS} between recurrent configurations on Ferrers graphs and the intransitive trees of Postnikov~\cite{Post}, to bijectively connect recurrent configurations of permutation graphs and the tiered trees introduced by Dugan et al.~\cite{DGGS}, of which the intransitive trees are a special case.

In~\cite{ABBS}, Aval et al. introduced the so-called complete non-ambiguous binary trees (\cnabs), which arise from certain 0/1 fillings of square Ferrers diagrams. We show that the set of minimal recurrent configurations on permutation graphs is in bijection with  \cnabs.  We then generalize the \cnabs, which have a canonical root vertex, to a multirooted version, which we show to be in bijection with all recurrent configurations on the corresponding permutation graphs.

We also show that our results extend those of Perkinson et al. \cite{PYY}, connecting 
parking functions and labeled spanning trees of threshold graphs, which are a subset of permutation graphs.


The paper is organized as follows. In Section~\ref{sec:defs} we recall necessary 
definitions and provide a link between tiered trees and spanning trees of permutation graphs. 
In Section~\ref{sec:main_bij} we exhibit a bijection between tiered trees and recurrent configurations of the ASM on permutation graphs. We show how the level statistic and canonical toppling of a recurrent configuration can be read from the corresponding tree, and interpret the level statistic as the external activity of the tree. This provides a new proof, in the case of permutation graphs, of the famous result linking the level polynomial of the ASM to the ubiquitous Tutte polynomial (see Proposition~\ref{pro:Tutte_level}).
In Section~\ref{sec:minrec_CNAB} we recall the definition of complete non-ambiguous binary trees introduced by Aval et al.~\cite{ABBS}, show that these are in bijection with the set of minimal recurrent configurations of the ASM and introduce a generalization that we show to correspond to all recurrent configurations. 
Finally, in Section~\ref{sec:specialisations} we study two special cases of permutation graphs, namely Ferrers graphs (corresponding to permutations with a single descent) and threshold graphs, and recover results from~\cite{DSSS} and ~\cite{PYY} respectively.

\section{Definitions and Preliminaries}\label{sec:defs}

For any positive integer $n$, we let $[n]:=\{1,\ldots,n\}$ and $\clsn$ be the set of permutations of~$[n]$.

\subsection{Permutation graphs}\label{sec:perm_graphs}

To a permutation $\pi = \pi_1 \cdots \pi_n \in \clsn$, we associate a graph $G_{\pi}$ as follows. The vertex set of $\pi$ is $[n]$ and the edges are the pairs $(\pi_i,\pi_j)$ such that $i<j$ and $\pi_i>\pi_j$, that is, $(i,j)$ is an inversion of $\pi$. Such a graph is called a \emph{permutation graph}.

A permutation $\pi \in \clsn$ is said to be \emph{indecomposable} if there exists no positive integer $k<n$ such that $\{\pi_1,\ldots,\pi_k\} = [k]$. The following is well known, see for example \cite[Lemma~3.2]{koh-reh-conn-perm-graphs}.

\begin{fact}
A permutation graph $G_{\pi}$ is connected if and only if $\pi$ is indecomposable.
\end{fact}

Figure~\ref{fig:example_permgraph} shows the graphs associated with the permutations $\pi = 23541$ and $\pi' = 23154$. Note that $\pi'$ can be decomposed as 231--54, while $\pi$ is indecomposable. Thus the graph $G_{\pi}$ is connected, while $G_{\pi'}$ is not.

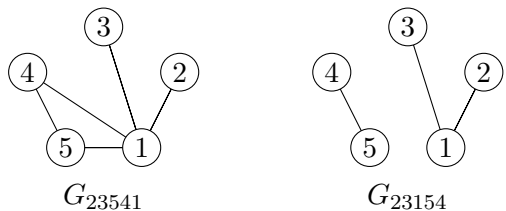
\begin{figure}[h]
  \centering
  \begin{tikzpicture}[scale=0.5]
  
  \draw (0,0)--(2,0)--(0,0)--(-1,2)--(2,0)--(1,3.2)--(2,0)--(3,2)--(2,0);
  \draw [fill=white] (0,0) circle [radius=0.5];
  \draw [fill=white] (-1,2) circle [radius=0.5];
  \draw [fill=white] (1,3.2) circle [radius=0.5];
  \draw [fill=white] (3,2) circle [radius=0.5];
  \draw [fill=white] (2,0) circle [radius=0.5];
  \node at (0,0) {$5$};
  \node at (-1,2) {$4$};
  \node at (1,3.2) {$3$};
  \node at (3,2) {$2$};
  \node at (2,0) {$1$};
  \node at (1,-1.3) {$G_{23541}$};
  
  \begin{scope}[shift={(8,0)}]
  \draw (0,0)--(-1,2);
  \draw (1,3.2)--(2,0)--(3,2)--(2,0);
  \draw [fill=white] (0,0) circle [radius=0.5];
  \draw [fill=white] (-1,2) circle [radius=0.5];
  \draw [fill=white] (1,3.2) circle [radius=0.5];
  \draw [fill=white] (3,2) circle [radius=0.5];
  \draw [fill=white] (2,0) circle [radius=0.5];
  \node at (0,0) {$5$};
  \node at (-1,2) {$4$};
  \node at (1,3.2) {$3$};
  \node at (3,2) {$2$};
  \node at (2,0) {$1$};
  \node at (1,-1.3) {$G_{23154}$};
  \end{scope}
  \end{tikzpicture}
  \caption{The graphs associated with the permutations $\pi = 23541$ (left) and $\pi' = 23154$ (right).\label{fig:example_permgraph}}
\end{figure}

Since we will be analyzing the ASM on permutation graphs, and the ASM is only defined on connected graphs, we will from now on only deal with permutation graphs of indecomposable permutations unless otherwise specified.

\subsection{Tiered trees}\label{sec:tiered_trees}

Tiered trees were introduced in~\cite{DGGS} as a generalization of the intransitive trees introduced by Postnikov~\cite{Post}, the latter of which have exactly two tiers. 

\begin{definition}
A \emph{tiered tree} of size $n$ is a pair $(T,t)$ where:
\begin{itemize}
\item $T$ is a labeled tree on $[n]$.
\item $t$ is a surjective mapping from $[n] \rightarrow [k]$ for some $k$ such that for any edge $(i,j)$ of $T$ with $i>j$ we have $t(i) < t(j)$.
\end{itemize}
The function $t$ is called the \emph{tiering} function of the tiered tree $(T,t)$, and the integer $k$ is its number of tiers.

A tiered tree is said to be \emph{fully tiered} if its number of tiers equals its number of vertices, that is, $k=n$, or equivalently, if its tiering function is a bijection.
\end{definition}

\begin{remark}
The condition $t(i)<t(j)$ is reversed in~\cite{DGGS}. This corresponds to replacing the function $t$ with $k+1-t$. The reason we reverse this condition is to make the link between tiered trees and permutation graphs simpler.
\end{remark}

\subsection{Fully tiered trees and permutation graphs}\label{sec:tieredtrees_permgraphs}

\begin{lemma}\label{lem:tiered_trees_fully_tiered}
Let $\T = (T,t)$ be a tiered tree. Then there exists a fully tiered tree $\T' = (T,t')$.
\end{lemma}
\begin{proof}
Let $\T=(T,t)$ be a tiered tree. For $\ell \in [k]$, we 
let $P_\ell := t^{-1}(\ell)$ be the set of vertices at tier $\ell$ in $\T$. By definition, the $P_\ell$ form a partition of $[n]$. We define $t':[n] \rightarrow [n]$ by 
\begin{equation}\label{eqn:def_t'}
t'(i):=\left(\sum_{m=1}^{\ell-1} \vert P_m \vert\right) + \vert \{j \in P_\ell : \, j < i\} \vert+1,
\end{equation} 
where $\ell$ is such that $i \in P_\ell$.
In words, the function $t'$ keeps the relative ordering of tiers, and orders vertices inside each tier in increasing order, as illustrated in Figure~\ref{fig:fully_tiered} below.  We claim that $t'$ is a tiering function for the tree $T$, and that $(T,t')$ is fully tiered.

\begin{figure}[h]
  \centering
  \begin{tikzpicture}[scale=0.7]
  
  \draw (-2,2)--(0,0)--(0,2)--(0,0)--(2,4)--(2,0);
  \draw[dotted] (-3,0)--(3,0);
  \draw[dotted] (-3,2)--(3,2);
  \draw[dotted] (-3,4)--(3,4);
  \draw [fill=white] (0,0) circle [radius=0.3];
  \draw [fill=white] (-2,2) circle [radius=0.3];
  \draw [fill=white] (0,2) circle [radius=0.3];
  \draw [fill=white] (2,4) circle [radius=0.3];
  \draw [fill=white] (2,0) circle [radius=0.3];
  \node at (0,0) {$5$};
  \node at (-2,2) {$4$};
  \node at (0,2) {$1$};
  \node at (2,4) {$2$};
  \node at (2,0) {$3$};
  \node [left] at (-3,0) {$t(\cdot)=1$};
  \node [left] at (-3,2) {$t(\cdot)=2$};
  \node [left] at (-3,4) {$t(\cdot)=3$};
  \node at (0,-1.3) {$\T$};
  
  \begin{scope}[shift={(6,0)}]
  \draw (0,1)--(0,2);
  \draw [out=135,in=225] (0,0) to (0,4);
  \draw [out=30,in=-30] (0,1) to (0,4);
  \draw [out=45,in=-45] (0,1) to (0,3);
  \draw[dotted] (-1,0)--(1,0);
  \draw[dotted] (-1,1)--(1,1);
  \draw[dotted] (-1,2)--(1,2);
  \draw[dotted] (-1,3)--(1,3);
  \draw[dotted] (-1,4)--(1,4);
  \draw [fill=white] (0,0) circle [radius=0.3];
  \draw [fill=white] (0,1) circle [radius=0.3];
  \draw [fill=white] (0,2) circle [radius=0.3];
  \draw [fill=white] (0,3) circle [radius=0.3];
  \draw [fill=white] (0,4) circle [radius=0.3];
  \node at (0,0) {$3$};
  \node at (0,1) {$5$};
  \node at (0,2) {$1$};
  \node at (0,3) {$4$};
  \node at (0,4) {$2$};
  \node [right] at (1,0) {$t(\cdot)=1$};
  \node [right] at (1,1) {$t(\cdot)=2$};
  \node [right] at (1,2) {$t(\cdot)=3$};
  \node [right] at (1,3) {$t(\cdot)=4$};
  \node [right] at (1,4) {$t(\cdot)=5$};
  \node at (0,-1.3) {$\T'$};
  \end{scope}
  \end{tikzpicture}
  \caption{A tiered tree $\T$ (left) and a fully tiered tree $\T'$ (right) with the same underlying tree. The tiers are represented as levels.\label{fig:fully_tiered}}
\end{figure}
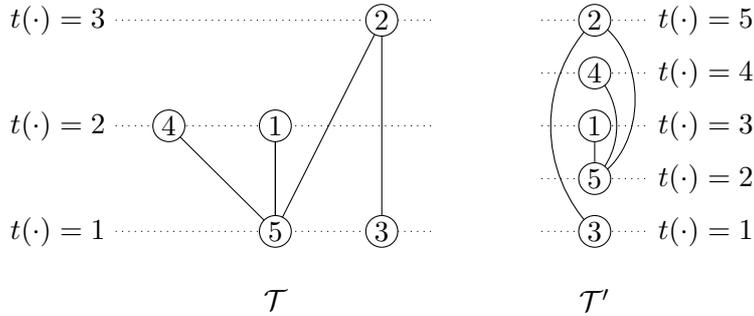

Let $(i,j)$ be an edge of $T$ with $i<j$, and let $\ell,m$ be such that $i \in P_\ell$ and $j \in P_m$. Since $t$ is a tiering function, this implies that $t(i) = \ell > m = t(j)$. Now by construction, Equation~\eqref{eqn:def_t'} implies that $t'(i) > t'(j)$, as desired.

It is clear from Equation~\eqref{eqn:def_t'} that $t'$ assigns a unique positive 
number no greater than $n$ to each $i$, which implies that $t'$ is a bijection, so $(T,t')$ is fully tiered.
\end{proof}

Lemma~\ref{lem:tiered_trees_fully_tiered} states that any tiered tree can be viewed as a fully tiered tree in a sense. As such, from now on, we only consider fully tiered trees, 
and call these simply tiered trees. The following proposition establishes a link between tiered trees and permutation graphs.
\begin{prop}\label{pro:tieredtrees_permgraphs}
Let $T$ be a labeled tree on $[n]$ and $\pi \in \clsn$. Then $T$ is a spanning tree of~$G_{\pi}$ if and only if $(T,\pi^{-1})$ is a tiered tree.
\end{prop}

\begin{proof}
Suppose that $T$ is a spanning tree of $G_{\pi}$. This means that if $(i,j)$ is an edge of $T$ with $i>j$, then $i$ appears before $j$ in $\pi$, which implies $\pi^{-1}(i)<\pi^{-1}(j)$. This is exactly the condition that $(T,\pi^{-1})$ is a tiered tree. The converse follows in the same way.
\end{proof}

\subsection{The Abelian sandpile model}\label{sec:ASM}

The ASM is a dynamic process on a graph which has attracted considerable attention through the years, and remains a constant source of new and interesting research topics. 

Let $G=(V,E)$ be a finite, connected, loop-free, undirected graph with vertex set  $V=[n]$ for some $n$.  Let $d_i = d_i(G)$ be the degree of the vertex $i$ in $G$.  We will consider the sandpile model on the graph $G$ with a distinguished vertex $s \in [n]$, called the \emph{sink}. We indicate that by writing this as the pair $\gas{G}{s}$.

A \emph{configuration} on $\gas{G}{s}$ is a vector $c=(c_1,\ldots ,c_n) \in \mathbb{Z}_+^n$ that assigns the number $c_i$ to vertex $i$.
We think of $c_i$ as the number of `grains of sand' at the vertex $i$. 
$\Config{G}$ is the set of all configurations on $\gas{G}{s}$.
Let $\alpha_i \in \mathbb{Z}^n$ be the vector with $1$ in the $i$-th position and~$0$ elsewhere. 

We say that a vertex $i$ is \emph{stable} in a configuration $c=(c_1,\ldots ,c_n)\in \Config{G}$ if $c_i < d_i$.  Otherwise it is \emph{unstable}.  A configuration is stable if all its non-sink vertices are stable.
 
Unstable vertices may \emph{topple}.  We define the toppling operator $T_i$ corresponding to the toppling of an unstable vertex $i \in [n]$ in a configuration $c \in \Config{G}$ by 
$$
T_i(c) := c - d_i \alpha_i + \sum_{j: \{i,j\} \in E} \alpha_j,
$$ 
where the sum is over all vertices adjacent to $i$.  In words, when a vertex $i$ topples, it sends one grain of sand along each incident edge to its neighbors. We write $c \ra{i} c'$ to indicate that the vertex $i$ is unstable in $c$ and that $T_i(c)=c'$.

It is possible to show (see for instance \cite[Section 5.2]{Dhar}) that starting from any configuration~$c$ and toppling unstable vertices, one eventually reaches a stable 
configuration $c'$.  Moreover, $c'$ does not depend on the order in which unstable vertices are toppled in this sequence.

\begin{definition}\label{def:rec_states}
A configuration $c \in \Config{G}$ is \emph{recurrent} on $\gas{G}{s}$ if it satisfies the following three conditions:
\begin{enumerate}
\item We have $c_s=d_s$.
\item The configuration $c$ is stable, that is, $c_i < d_i$ for $i \neq s$.
\item\label{defrec3} There exists a sequence $v_1,\ldots,v_n$ with $v_1=s$ and $\{v_1,\ldots,v_n\} = [n]$ such that 
$$c^0 = c \ra{v_1} c^1 \ra{v_2} \cdots \ra{v_n} c^n=c.$$
\end{enumerate}
\end{definition}

In words, the third condition states that there is an ordering of the vertices such that starting from $c$, every vertex can be toppled (exactly) once in this order.  The fact that after making these topplings one returns to the configuration $c$ is guaranteed by the following argument:  On every edge $(i,j)$ of $G$, toppling $i$ sends one grain from $i$ to $j$ while toppling~$j$ sends one grain from $j$ to $i$. Thus, toppling every vertex exactly once leaves the initial configuration unchanged.

Let $\Rec{s}{G}$ be the set of recurrent configuration on a graph $G$ with sink $s$. 
Given $c\in\Rec{s}{G}$, define the \emph{level} of $c$ to be 
$$
\level{c} := \sum\limits_{i \in [n]} c_i - |E|,
$$
where $|E|$ denotes the number of edges of $G$.  From \cite[Thm. 3.5]{Lop} we have that if $G=(V,E)$ is a graph and $c \in \Rec{s}{G}$, then $0\leq\level{c}\leq |E|-|V|+1$.
The level of a recurrent configuration is thus always a non-negative integer. The \emph{level polynomial} of a graph $\gas{G}{s}$ is the generating function of the level statistic over the set of recurrent configurations on that graph:
$$
\Levelpoly{G}{s}{x} := \sum\limits_{c \in \Rec{s}{G}} x^{\level{c}}.
$$

Finally, we define the notion of \emph{canonical toppling}.  Given a recurrent configuration $c \in \Rec{s}{G}$, the canonical toppling of $c$ is the ordered set partition $P=P_0,\ldots,P_k$ of $[n]$ 
where $P_0=\{s\}$ and for $i \geq 1$, $P_i$ is the set of (non-sink) unstable vertices resulting from the toppling of all vertices in $P_0,\ldots,P_{i-1}$.  The fact that this gives a partition of $[n]$ is guaranteed by Condition~\eqref{defrec3} of Definition~\ref{def:rec_states}.  For $c \in \Rec{s}{G}$, we denote by $\CanonTop(c)$ the canonical toppling of $c$.

\begin{example}\label{ex:rec_canontop}
Let $\pi=3421$ and $G_{3421}$ be the corresponding permutation graph, as illustrated in Figure~\ref{fig:ex_canontop}. 
Fix $s=3$ to be the sink vertex (represented as a square), and consider the configuration $c=(1,2,2,1)$ (grains are represented as red dots next to their vertex).
We have $c_3=2=d_3$ and $c_i < d_i$ for $i \neq 3$ so the first two conditions of Definition~\ref{def:rec_states} are satisfied. We show that the third condition is also satisfied, and simultaneously determine the canonical toppling.

We initially topple vertex $3$ in $c^0=c$. This yields the configuration $c^1 = (2,3,1,0)$. In~$c^1$, only vertex $2$ is unstable, so we topple this, reaching $c^2=(3,0,1,2)$. Now both $1$ and $4$ are unstable.
 In this case, we may topple for instance $1$ then $4$, and this will yield the initial configuration $c$. Thus, $c$ is recurrent and $\CanonTop(c) = \{3\},\{2\},\{1,4\}$. Finally, we can compute the level of 
$c$: $\level{c} = 1+2+2+1-5 = 1$.

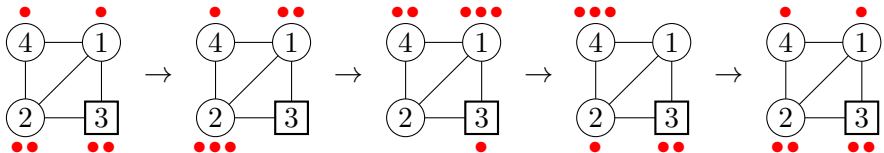
\begin{figure}[h]
  \centering
  \begin{tikzpicture}[scale=0.5]
\def\graph{
\draw (0,0)--(2,0)--(2,2)--(0,2)--(0,0)--(2,2);
  \draw [fill=white] (0,0) circle [radius=0.5];
  \draw [fill=white] (0,2) circle [radius=0.5];
  \draw [fill=white] (2,2) circle [radius=0.5];
  \draw [thick, fill=white] (1.57,-0.43) rectangle (2.43,0.43);
  \node at (0,0) {$2$};
  \node at (0,2) {$4$};
  \node at (2,0) {$3$};
  \node at (2,2) {$1$};
}
\newcommand\rdot{\setlength\unitlength{1mm}\red{\circle*{1.5}}}
\newcommand\rdotsb{\makebox(0,0){\rdot\,\rdot}}
\newcommand\rdotsc{\makebox(0,0){\rdot\,\rdot\,\rdot}}
\newcommand{\grainsa}[2]{\node at (#1,#2) {\rdot};}
\newcommand{\grainsb}[2]{\node at (#1,#2) {\rdotsb};}
\newcommand{\grainsc}[2]{\node at (#1,#2) {\rdotsc};}
\newcommand\arrow{\node at (-1.5,1.1) {$\rightarrow$};}

\graph
  \grainsa{2.13}{2.8}
  \grainsb{0.13}{-0.78}
  \grainsb{2.13}{-0.78}
  \grainsa{0.13}{2.8}

\begin{scope}[shift={(5,0)}]{ 
\graph
\arrow
  \grainsb{2.13}{2.8}
  \grainsc{0.13}{-0.78}
  \grainsa{0.13}{2.8}
};
\end{scope}

\begin{scope}[shift={(10,0)}]{ 
\arrow
\graph
  \grainsc{2.13}{2.8}
  \grainsa{2.13}{-0.78}
  \grainsb{0.13}{2.8}
};
\end{scope}

\begin{scope}[shift={(15,0)}]{ 
\arrow
\graph
  \grainsa{0.13}{-0.78}
  \grainsb{2.13}{-0.78}
  \grainsc{0.13}{2.8}
};
\end{scope}

\begin{scope}[shift={(20,0)}]{ 
\arrow
\graph
  \grainsa{2.13}{2.8}
  \grainsb{0.13}{-0.78}
  \grainsb{2.13}{-0.78}
  \grainsa{0.13}{2.8}
};
\end{scope}

\end{tikzpicture}
  \caption{The permutation graph $G_{3421}$ with sink $s=3$ and the configuration $c=(1,2,2,1)$, which is shown to be recurrent by the toppling sequence 3,2,1,4. \label{fig:ex_canontop}}
\end{figure}
\end{example}

\section{A bijection from trees to recurrent configurations of the ASM}\label{sec:main_bij}

\subsection{The bijection}\label{sec:main_thm}

Let $\pi \in \clsn$ and $T$ be a spanning tree of $G=G_{\pi}$, that is, such that $(T,\pi^{-1})$ is a tiered tree by Proposition~\ref{pro:tieredtrees_permgraphs}. 
Let $s \in [n]$ be a distinguished vertex of $G$. We view the tree $T$ as being rooted at $s$. Given $i \in [n]$, we define the \emph{height} of $i$ in $T$ to be its distance to the root $s$, and denote it $h(i)$. If $i \neq s$, the \emph{parent} of $i$ is the next vertex encountered on the unique path from $i$ to $s$, and we denote this $p(i)$. For  $k \geq 1$, we define $T^{(k)} := \{ i \in n: \, h(i)=k\}$ to be the set of vertices at height $k$ 
in $T$, with analogous definitions for $T^{(>k)}$, $T^{(\geq k)}$, etc. Finally, we let $N_G(i)$ be the set of neighbors of $i$ in the graph $G$.

For $i \in [n]$, we let:
\begin{eqnarray}
\lambda_i = \lambda_i(T) & := & \left\vert N_G(i) \cap T^{\left(>h(i)\right)} \right\vert \label{eq:def_lambda} ,\\
\mu_i = \mu_i(T) & := & \left\vert N_G(i) \cap T^{\left(h(i)\right)} \right\vert \label{eq:def_mu} ,\\
\nu_i = \nu_i(T) & := & \left\vert N_G(i) \cap T^{\left(h(i) - 1\right)} \cap [0,p(i)-1] \right\vert \label{eq:def_nu} .
\end{eqnarray}

In words, $\lambda_i$ is the set of neighbors of $i$ in $G$ at height strictly greater than $i$ in $T$, $\mu_i$ is the set of neighbors of $i$ in $G$ at the same height as $i$ in $T$, and $\nu_i$ is the set of neighbors of~$i$ in $G$ at height one less than $i$, and whose labels are strictly smaller than the parent of $i$. Although it would be  natural to combine $\lambda_i$ and $\mu_i$ into one number, this definition facilitates our proof of the following theorem. Note that these definitions all depend on the choice of a distinguished vertex $s$, though for lightness of notation we do not make this explicit.

\begin{theorem}\label{thm:main_result}
Let $\pi \in \clsn$ be a permutation and $s \in [n]$ a distinguished vertex of $G=G_{\pi}$.\linebreak Given a spanning tree $T$ of $G$ we define a configuration $c(T)=(c_1(T),\ldots,c_n(T)) \in \Config{G}$ by
$$
c_i(T) := \lambda_i(T) + \mu_i(T) + \nu_i(T).
$$
Then the map $\ttc : T \mapsto c(T)$ is a bijection from the set of spanning trees of $G$ to $\Rec{s}{G}$.

Moreover, for any spanning tree $T$, we have $ \level{c(T)} = \sum\limits_{i=1}^n \left( \frac12 \mu_i(T) + \nu_i(T) \right) $, and 
$$
\CanonTop \left( c(T) \right) = T^{(0)},T^{(1)},\ldots.
$$ 
That is, the canonical toppling of $c(T)$ is given by the breadth-first search of $T$.
\end{theorem}

Before we prove this result, let us examine one example in depth.

\begin{example}
Let $\pi = 514362$. The associated permutation graph $G=G_{\pi}$ is represented on the left of Figure~\ref{fig:ex_bij}. We take $s=3$ to be the sink. Let $T$ be the spanning tree of $G$ on the right of Figure~\ref{fig:ex_bij}. We represent $T$ as a tree rooted at the distinguished vertex $3$, and compute the corresponding configuration $c(T)$: 

For $i=1$, there are no vertices at height greater than $h(1)=2$ in $T$, and none of the other two vertices at height $2$ are neighbors of $1$ in $G$, so that $\lambda_1=\mu_1=0$. In fact, the parent of $1$ in $T$ is its only neighbor in $G$, so that we also have $\nu_1=0$, and thus $c_1=0$. Now consider the vertex $i=2$. In $T$
there are three vertices at height greater than $h(2)=1$, which are $1,4,6$. Of these, $4$ and~$6$ are neighbors of $2$ in G, so that $\lambda_2=2$. Similarly, $5$ is the other vertex at height $1$ in $T$, and is a neighbor of $2$ in $G$, so that $\mu_2=1$. Finally, the parent of $2$ in $T$ is the only vertex at height $1-1=0$, so $\nu_2=0$. Thus, $c_2=2+1+0=3$.

Similarly, $c_3=3+0+0=3$. Now for $i=4$, we have $\lambda_4=0$ (there are no vertices at 
a greater height in $T$), $\mu_4=0$ (neither $1$ nor $6$ are neighbors of $4$ in $G$). But both $2$ and $5$ are neighbors of $4$ in $G$ with height equal to $h(4)-1$ in $T$, and the parent of $4$ in $T$ is $5$, so that $\nu_4=1$, and thus $c_4=0+0+1=1$. Finally, we can see that $c_5=2+1+0=3$, and $c_6=0+0+0=0$. Thus, we have $c(T)=(0,3,3,1,3,0)$.

We check that $c(T)$ is recurrent using Definition~\ref{def:rec_states}, and also establish that the canonical toppling of $c(T)$ is given by the breadth-first search (BFS) of $T$. The vertex degree sequence of $G$ is given by $(1,4,3,3,4,1)$, and the BFS of $T$ is $3-25-146$ with dashes separating the sets of vertices at different heights. Start from the configuration $c=c(T)=(0,3,3,1,3,0)$. We have $c_3=d_3$ and $c_j<d_j$ for $j \neq 3$, as desired. Therefore we initially topple vertex $3$. This leads to the configuration $(0,4,0,2,4,0)$. In this configuration, vertices $2$ and $5$ are unstable. We topple these, which leads to the configuration $(1,1,2,4,1,1)$. In this configuration, vertices $1$,$4$ and~$6$ are unstable. We topple these, which leads back to the initial configuration $(0,3,3,1,3,0)$. Thus, by Definition~\ref{def:rec_states} the configuration $c(T)$ is recurrent, and we have moreover shown that $\CanonTop\left(c^T \right) = 3-25-146$, which is exactly the BFS of $T$.

Finally, the graph $G$ has $8$ edges, so that on the one hand 
$$\level{c(T)} = (0+3+3+1+3+0) - 8 = 10 - 8 =2.$$ 
On the other hand, we have
$$\sum\limits_{i=1}^6 \left( \frac12 \mu_i + \nu_i \right) = \frac12 (0+1+0+0+1+0) + (0+0+0+1+0+0) = 1+1 = 2,$$ 
which gives the desired result.

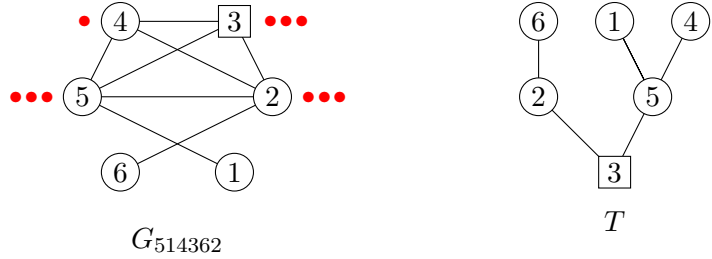
\begin{figure}[h]
\centering
\begin{tikzpicture}[scale=0.5]
  \newcommand\vertex[2]{\draw [fill=white] (#1,#2) circle [radius=0.5];}
  \newcommand\rdot{\setlength\unitlength{1mm}\red{\circle*{1.5}}}
  \newcommand\rdotsb{\makebox(0,0){\rdot\,\rdot}}
  \newcommand\rdotsc{\makebox(0,0){\rdot\,\rdot\,\rdot}}
  \newcommand{\grainsa}[2]{\node at (#1,#2) {\rdot};}
  \newcommand{\grainsb}[2]{\node at (#1,#2) {\rdotsb};}
  \newcommand{\grainsc}[2]{\node at (#1,#2) {\rdotsc};}

  \draw (0,0)--(4,2);
  \draw (-1,2)--(3,0);
  \draw (-1,2)--(4,2)--(3,4)--(0,4)--(-1,2)--(3,4);
  \draw (0,4)--(4,2);
   \vertex{0}{0};
   \vertex{-1}{2};
   \vertex{0}{4};
   \vertex{4}{2};
\draw [fill=white] (2.58,3.58) rectangle ++(0.84,0.84);
   \vertex{3}{0};

  \node at (0,0) {$6$}; 
  \node at (-1,2) {$5$}; \grainsc{-2.2}{2}
  \node at (0,4) {$4$}; \grainsa{-0.8}{4}
  \node at (3,4) {$3$}; \grainsc{4.5}{4}
  \node at (4,2) {$2$}; \grainsc{5.5}{2}
  \node at (3,0) {$1$};
  \node at (1.5,-1.8) {$G_{514362}$};
  
  \begin{scope}[shift={(13,0)}]
  \draw (0,0)--(-2,2)--(-2,4);
  \draw (0,0)--(1,2)--(0,4)--(1,2)--(2,4);
  \draw [fill=white] (-0.42,-0.42) rectangle (0.42,0.42);
  \draw [fill=white] (-2,2) circle [radius=0.5];
  \draw [fill=white] (-2,4) circle [radius=0.5];
  \draw [fill=white] (1,2) circle [radius=0.5];
  \draw [fill=white] (0,4) circle [radius=0.5];
  \draw [fill=white] (2,4) circle [radius=0.5];

  \node at (0,0) {$3$};
  \node at (-2,2) {$2$};
  \node at (-2,4) {$6$};
  \node at (1,2) {$5$};
  \node at (0,4) {$1$};
  \node at (2,4) {$4$};
  \node at (0,-1.3) {$T$};
  \end{scope}
  
\end{tikzpicture}
\caption{The graph $G$ associated with the permutation $\pi = 514362$ (left) and a spanning tree $T$ of $G$ represented as rooted at the distinguished vertex~$3$ (right). Configuration on $G$ corresponding to $T$ shown with red dots.\label{fig:ex_bij}}
\end{figure}
\end{example}

\begin{proof}[Proof of Theorem \ref{thm:main_result}]
Let $T$ be a spanning tree of $T$, and $c:=c(T)$ the corresponding configuration. We first show that $c$ is recurrent, and that $\CanonTop(c) = T^{(0)},T^{(1)},\ldots$, using Definition~\ref{def:rec_states}.
\begin{enumerate}
\item The sink $s$ is the unique vertex at height $0$ in $T$, so that $\lambda_s = \vert N_G(s) \vert = d_s$, and $\mu_s = \nu_s = 0$. Thus $c_s = \lambda_s + \mu_s + \nu_s = d_s$ as desired.
\item For $i \neq s$, we see that $\lambda_i$, $\mu_i$ and $\nu_i$ all count distinct subsets of $N_G(i)$. Moreover, $p(i)$ is a neighbor of $i$ in $G$ which is counted in none of these three subsets. Thus, $c_i < \vert N_G(i) \vert = d_i$, and so the configuration $c$ is stable.
\item We now show that, starting from the configuration $c$, for any $k \geq 1$, if we topple the vertices of $T^{(0)},\ldots,T^{(k-1)}$, then the set of non-sink unstable vertices is exactly $T^{(k)}$. Combined with the above, this shows that $c$ is recurrent, 
and that $\CanonTop(c) = T^{(0)},T^{(1)},\ldots$. Let $k \geq 1$ and let $c'$ be the configuration reached from the initial configuration $c$ after toppling the vertices of $T^{(0)},\ldots,T^{(k-1)}$. We need to show that $c'_i \geq d_i$ if $i \in T^{(k)}$, and that $c'_i < d_i$ if $i \notin T^{(k)} \cup \{s\}$.
\begin{itemize}
\item Let $i \in T^{(k)}$. We have $c'_i = c_i + \left\vert N_G(i) \cap T^{(<k)} \right\vert$, since the second term of the sum is the number of grains vertex $i$ receives through toppling $T^{(0)},\ldots,T^{(k-1)}$. Thus
\begin{align*}
c'_i & = \lambda_i + \mu_i + \nu_i + \left\vert N_G(i) \cap T^{(<k)} \right\vert \\
 & = \left\vert N_G(i) \cap T^{(>k)} \right\vert + \left\vert N_G(i) \cap T^{(k)} \right\vert + \nu_i + \left\vert N_G(i) \cap T^{(<k)} \right\vert \\
 & = d_i + \nu_i \geq d_i,
\end{align*} 
as desired.
\item Let $i \in T^{(>k)}$. Write $ \ell = h(i) > k$. As above, we have 
\begin{align*}
c'_i & = c_i + \left\vert N_G(i) \cap T^{(<k)} \right\vert \\
 & = \left\vert N_G(i) \cap T^{(> \ell)} \right\vert + \left\vert N_G(i) \cap T^{(\ell)} \right\vert + \left\vert N_G(i) \cap T^{(<k)} \right\vert + \nu_i. \\
\end{align*} 
Now $\nu_i$ counts a subset of neighbors of $i$ in $G$ which are at height $\ell-1$ in $T$, and since $p(i)$ is not counted in $\nu_i$, this is a strict subset.
Thus $c'_i < \left\vert N_G(i) \cap T^{(> \ell)} \right\vert + \left\vert N_G(i) \cap T^{(\ell)} \right\vert + \left\vert N_G(i) \cap T^{(<k)} \right\vert + \left\vert N_G(i) \cap T^{(\ell -1)} \right\vert$, and since $\ell-1 \geq k$, it follows that $c' <d_i$ as desired.
\item Finally, let $i \in T^{(<k)}$, with $i \neq s$. The vertex $i$ has been toppled in $T^{(0)},\ldots,T^{(k-1)}$, so that $c'_i = c_i + \left\vert N_G(i) \cap T^{(<k)} \right\vert - d_i$. But we have already shown that $c$ is stable, so $c_i <d_i$, and thus $c'_i < \left\vert N_G(i) \cap T^{(<k)} \right\vert \leq \vert N_G(i) \vert = d_i$, as desired.
\end{itemize}
\end{enumerate}
This completes the first part of the proof, namely that $c$ is recurrent, and that $\CanonTop(c) = T^{(0)},T^{(1)},\ldots$.

We now show that $ \level{c} = \sum\limits_{i=1}^n \left( \frac12 \mu_i + \nu_i \right)$. We have
\begin{align*}
\level{c} & = \sum\limits_{i=1}^n \left( \lambda_i + \mu_i + \nu_i \right) - \vert E \vert \\
 & =  \sum\limits_{i=1}^n \left( \lambda_i + \frac12 \mu_i \right) - \vert E \vert + \sum\limits_{i=1}^n \left( \frac12 \mu_i + \nu_i \right).
\end{align*}
Now, the sum $\sum_{i=1}^n \lambda_i$ counts all pairs of vertices $(i,j)$ such that $j \in N_G(i)$ and $h(i) < h(j)$. Thus every edge $(i,j)$ of $G$ with $h(i) \neq h(j)$ is counted exactly once in that sum. Moreover, the sum $\sum_{i=1}^n \mu_i$ counts all pairs of vertices $(i,j)$ such that $j \in N_G(i)$ and $h(i) = h(j)$. Thus, in the sum $\sum_{i=1}^n \mu_i$, every edge $(i,j)$ of $G$ with $h(i) = h(j)$ is counted twice. Therefore we have $\sum_{i=1}^n \left( \lambda_i + \frac12 \mu_i \right) = \vert E \vert $, and thus $ \level{c} = \sum_{i=1}^n \left( \frac12 \mu_i + \nu_i \right)$, as desired.


It remains to show that $\ttc$ is a bijection. To do this, we exhibit its inverse. Let $c \in \Rec{s}{G}$, and write $\CanonTop(c) = P_0,P_1,\ldots$ for the canonical toppling of $c$, with $P_0 = \{s\}$. We construct a spanning tree $T = T(P)$ of $G$ from this as follows. The levels of $T$ are such that for all $j \geq 0$ we have $T^{(j)} = P_j$. To define $T$ it is then sufficient to define a parent map $p:[n] \setminus \{s\} \rightarrow [n]$ such that for any $j \geq 1$ and $i \in P_j$, we have $p(i) \in N_G(i) \cap P_{j-1}$. That this intersection is nonempty follows from the definition of the canonical toppling, since for $i$ to topple in $P_j$ it must have received some grains through the toppling of $P_{j-1}$.

Fix some $j \geq 1$ and $i \in P_j$. The definition of the canonical toppling implies the following property: Starting from $c$, the vertex $i$ is stable after toppling the vertices from $P_0,\ldots,P_{j-2}$, and becomes unstable after toppling those of $P_{j-1}$. For $k \geq 0$, let $N_G^{(<k)}(i)$ be the set of neighbors of $i$ in $G$ which are in $P_0 \cup \cdots \cup P_{k-1}$. The previous property can then be summarized in the following two inequalities:
$$ c_i + \left\vert N_G^{(<j-1)}(i) \right\vert < d_i, $$
$$ c_i + \left\vert N_G^{(<j)}(i) \right\vert \geq d_i. $$
Letting $r_i := c_i + \left\vert N_G^{(<j)}(i) \right\vert - d_i$, this is equivalent to
$$ 0 \leq r_i < \left\vert N_G(i) \cap P_{j-1} \right\vert.$$
We then define $p(i)$ to be the $(r_i+1)$-th largest element of $ N_G(i) \cap P_{j-1} $, 
and let $T=\ctt(c)$ be the spanning tree of $G$ resulting from this construction. We now show that $\ctt$ is the inverse of $\ttc$.

First, let $T$ be a spanning tree of $G$ and set $T':= \ctt(\ttc(T))$. By construction, 
we have $T^{(k)} = T'^{(k)}$ for all $k \geq 0$, so we only need to show that for any $i \in [n] \setminus \{s\}$, we have $p^T(i) = p^{T'}(i)$. Set $c:=\ttc(T)$ and let $i \in T^{(j)} \left( = T'^{(j)} \right) $ for some $ j \geq 1$. By definition, $p^{T'}(i)$ is the $(r_i+1)$-th largest element of $ N_G(i) \cap T^{(j-1)} $, where $r_i := c_i + \left\vert N_G^{(<j)}(i) \right\vert - d_i = c_i - \left\vert N_G(i) \cap T^{( \geq j)} \right\vert$. But by definition of $\ttc$, this means that $r_i = \nu_i(T) = \left\vert N_G(i) \cap T^{(j-1)} \cap [0,p(i) - 1] \right\vert$, and thus $p^T(i)$ is also the $(r_i+1)$-th largest element of $ N_G(i) \cap T^{(j-1)} $, so that $p^T(i) = p^{T'}(i)$ as desired. This shows that $\ctt(\ttc(T)) = T$ for any spanning tree $T$. Since it is well known that the number of recurrent configurations for the ASM on a graph $G$ is equal to a number of spanning trees of $G$ (see for instance~\cite[Section 3.2]{Red}), this is sufficient to conclude that $\ttc$ is a bijection, with $\ctt$ its inverse. 
\end{proof}

\begin{remark}\label{rem:bij_tiered_trees_rec_configs}
Theorem~\ref{thm:main_result}, combined with Proposition~\ref{pro:tieredtrees_permgraphs}, provides a bijection between the set of (fully) tiered trees on $[n]$ and the (disjoint) union of the sets of recurrent configurations for the ASM over all (connected) permutation graphs on $n$ vertices. In particular, we have that the number of (fully) tiered trees on $[n]$ is given by the sum $ \sum_{\pi} \vert \Rec{s}{G_{\pi}} \vert$, where the sum is over all indecomposable permutations of length $n$, and $s$ is some fixed (but arbitrary) sink in $[n]$.
\end{remark}

\subsection{A Tutte-descriptive activity}\label{sec:Tutte}

Let $\pi \in \clsn$ be a permutation, $G=G_{\pi}$ its permutation graph, and $s \in [n]$ a distinguished vertex of $G$. Given a spanning tree $T$, we interpret the level statistic of the corresponding recurrent configuration $c(T) \in \Rec{s}{G}$ as the external 
activity of the spanning tree $T$.  

\begin{definition}\label{def:activity}
Let $G=(V,E)$ be a graph, $T$ a spanning tree of $G$, and $\prec$ a total order on the edges $E$ of $G$. An edge $e \notin T$ is said to be \emph{externally active} if it is the maximal edge for $\prec$ in the unique cycle contained in $T \cup \{e\}$. An edge $e \in T$ is said to be \emph{internally active} if it is the maximal edge for $\prec$ in the unique \emph{cocycle} contained in $T \setminus \{e\}$, that is, in the set of edges connecting the two connected components of $T \setminus \{e\}$. The external, resp. internal, activity of $T$ is its number of externally, resp. internally, active edges, and is denoted by $\Ext(T)$, resp. $\Int(T)$.
\end{definition}

In light of this, Theorem~\ref{thm:main_result} can be interpreted as a bijection between recurrent configurations and spanning trees of a graph, mapping the level of a configuration to the external activity of the corresponding tree. This bijection is different from those already existing in the literature, such as~\cite{Ber,CLB}.

Recall that the Tutte polynomial of a (connected) graph $G=(V,E)$ is defined by
$$
\tut_G(x,y) := \sum\limits_{S \subseteq E} (x-1)^{\mathrm{cc}(S) -1} (y-1)^{\mathrm{cc}(S) + \vert S \vert - \vert V \vert},
$$
where for $S \subseteq E$, $\mathrm{cc}(S)$ denotes the number of connected components of the subgraph $(V,S)$. The level and Tutte polynomial of a graph are related by the following well-known result.

\begin{prop}\label{pro:Tutte_level}
Let $\gas{G}{s}$ be a graph. Then we have
$\Levelpoly{G}{s}{x} = \tut_G(1,x).$ In particular, the level polynomial is independent of the choice of sink.
\end{prop}

This result was initially proved by L\'opez \cite{Lop}, following a conjecture by Biggs. 
Subsequent combinatorial (bijective) proofs have been given, for instance, by Cori and Le Borgne~\cite{CLB}, and Bernardi~\cite{Ber}. The aim of the remainder of this section is to show that Theorem~\ref{thm:main_result} gives a new bijective proof in the case of permutation graphs.

Let $G=G_{\pi}$ be a permutation graph with sink $s$. We first show that for a spanning tree $T$ of $G$, we can construct an order $\prec_T$ on the edges of $G$ such that $\level{\ttc(T)} = \Ext(T)$. We then show that the order map $T \mapsto \prec_T$ is \emph{Tutte-descriptive} in the sense introduced by Courtiel in \cite{Cour}.
Let $T$ be a spanning tree of $G$. As usual, we root $T$ at $s$. The following algorithm defines an order $\prec_T$ of $E$.

\begin{algorithm}\label{algo:<_T}
\begin{enumerate}
\item Initially, set $k=0$ and all vertices as unvisited.
\item Let $v$ be the largest unvisited vertex at height $k$ in $T$. If no such vertex exists, increase $k$ by $1$ and repeat this step.
\item Let $S$ be the set of edges $(v,w)$ of $G$ such that $w$ is unvisited. Order elements of $S$ by $(v,w) \prec_T (v,w')$ if $w>w'$, and such that all edges in $S$ are greater (in $\prec_T$) than all previously ordered edges.
\item Mark $v$ as visited. If all edges of $G$ have been ordered then terminate, otherwise return to Step (2).
\end{enumerate}
\end{algorithm}

This order $\prec_T$ is similar to that introduced by Gessel and Sagan in \cite{GS}, though where theirs is based on a depth-first search of $T$ ours is based on a breadth-first search, since vertices are visited in that order. 

\begin{example}\label{ex:<_T}
Let $\pi = 514362$ so that $G_{\pi}$ is the graph on the left in Figure~\ref{fig:ex_bij}, and consider the spanning tree $T$ on the right in that figure.
We initially set $k=0$ and $v=3$ which is the only vertex at height $0$ in $T$. Proceeding to Step~(3), we have $S = \{(3,2),(3,4),(3,5)\}$. We order these $(3,5) \prec (3,4) \prec (3,2)$. We then mark $3$ as visited, and return to Step~(2). Since there are no unvisited vertices left at height $0$, we move to height $1$.

We set $v=5$, which is the largest vertex at height $1$ (neither vertex has been visited yet). Now $S=\{(5,1),(5,2),(5,4)\}$ since $3$ has already been visited, and we order these $(5,4) \prec (5,2) \prec (5,1)$, with $(3,2) \prec (5,4)$. We then mark $5$ as visited, return to Step~(2), and set $v=2$. We have $S=\{ (2,4),(2,6) \}$, which we order $(2,6) \prec (2,4)$, with $(5,1) \prec (2,6)$. We then mark $2$ as visited, and we now see that all edges have been ordered, so the algorithm terminates, and yields the order $(3,5) \prec (3,4) \prec (3,2) \prec (5,4) \prec (5,2) \prec (5,1) \prec (2,6) \prec (2,4)$.
\end{example}

\begin{theorem}\label{thm:ext_activity_level}
Let $G=G_{\pi}$ be a permutation graph with sink $s$. Then for any spanning tree $T$ of $G$, we have
$$ \Ext(T) = \level{\ttc(T)}, $$
where $\Ext(T)$ is the number of externally active edges for the order $\prec_T$ defined by Algorithm~\ref{algo:<_T}.
\end{theorem}

To prove this result, we need two lemmas.

\begin{lemma}\label{lem:level}
Let $G=G_{\pi}$ be a permutation graph with sink $s$, and $T$ a spanning tree of $G$. Then
\begin{align*}
\level{\ttc(T)} & = \left\vert \{ (i,j) \in E(G) : \, h(i) = h(j) \} \right\vert \\
 & + \left\vert \{ (i,j) \in E(G) : \, h(i) = h(j) -1 \text{ and } i<p(j) \} \right\vert.
\end{align*}
\end{lemma}

\begin{proof}
From Theorem~\ref{thm:main_result}, we have that $ \level{\ttc(T)} = \sum\limits_{i=1}^n \left( \frac12 \mu_i + \nu_i \right)$. Moreover, we saw in the proof of that result that $\sum\limits_{i=1}^n \frac12 \mu_i$ counts edges $(i,j)$ of $G$ such that $h(i) = h(j)$, that is the first term of the right-hand side of Lemma~\ref{lem:level}, while it is clear that $\sum\limits_{i=1}^n \nu_i$ counts the second term, so the result immediately follows.
\end{proof}

\begin{lemma}\label{lem:ext_active_edges}
Let $G=G_{\pi}$ be a permutation graph with sink $s$, $T$ a spanning tree of $G$, and~$\prec_T$ the order on the edges of $G$ given by Algorithm~\ref{algo:<_T}. Suppose that an edge $e=(i,j)$ is externally active for $\prec_T$. Then we have $ \vert h(i) - h(j) \vert \leq 1$.
\end{lemma}

\begin{proof}
Suppose that $e=(i,j)$ with $h(i) \geq h(j) + 2$. In particular, we have $h(i) > h(p(i)) > h(j)$. By the construction in Algorithm~\ref{algo:<_T}, we therefore have $(i,j) \prec_T (i,p(i))$, and since the unique cycle of $T \cup \{e\}$ contains the edge $(i,p(i))$ this implies that $e$ is not externally active, which completes the proof.
\end{proof}

We now prove Theorem~\ref{thm:ext_activity_level}.

\begin{proof}[Proof of Theorem~\ref{thm:ext_activity_level}]
Let $e=(i,j)$ be an edge of $G \setminus T$, with $h(i) \leq h(j)$. By Lemma~\ref{lem:level}, it is sufficient to show that $e$ is externally active if and only if $h(i) = h(j)$ or $h(i) = h(j) - 1$ and $i < p(j)$.

First suppose that $e$ is externally active. Lemma~\ref{lem:ext_active_edges} implies that we have  $h(i) = h(j)$ or $h(i) = h(j) - 1$. If $h(i) = h(j)$ there is nothing to do. If $h(i) = h(j) - 1$, we need to show that $i < p(j)$. But if $i > p(j)$ (we cannot have $i = p(j)$ since $(i,j)$ is not an edge of $T$) the construction in Algorithm~\ref{algo:<_T} implies that $(i,j) \prec_T (p(j),j)$ . Since the edge $(p(j),j)$ is contained in the unique cycle of $T \cup \{(i,j)\}$, this means that $e$ is not externally active, which is a contradiction. Hence we must have $i < p(j)$, as desired.

Conversely, suppose that $h(i) = h(j)$ or $h(i) = h(j) - 1$ and $i < p(j)$. Note that the unique cycle of $T \cup \{(i,j)\}$ is formed of the union of the paths $i\leftrightarrow i \wedge j$ and $j \leftrightarrow i \wedge j$ and of the edge $e$, where $i \wedge j$ is the greatest common ancestor of $i$ and $j$ in the tree $T$. If $h(i) = h(j)$ or $h(i) = h(j) - 1$, then all vertices of those paths other than $i$ and $j$ are visited before $i$ and $j$ in the construction of Algorithm~\ref{algo:<_T}, which implies that all edges of the paths $i\leftrightarrow i \wedge j$ and $j \leftrightarrow i \wedge j$ are ordered in $\prec_T$ before $(i,j)$, and thus that edge is externally active by definition. This completes the proof.
\end{proof}

Theorem~\ref{thm:ext_activity_level} states that the level of the configuration corresponding to a spanning tree $T$ via Theorem~\ref{thm:main_result} can be interpreted as the external activity of $T$ for a specific order $\prec_T$ of the edges of $G$. We now show that this order is Tutte-descriptive in the sense introduced by Courtiel \cite{Cour}.

\begin{definition}\label{def:Tutte_descriptive}
Let $G = (V,E)$ be a graph, and suppose we have a mapping $ \Psi: T \mapsto \prec_T$ from the set of spanning trees of $G$ to the set of total orders on $E$. We say that the mapping~$\Psi$ is \emph{Tutte-descriptive} if
$$
\tut_G(x,y) = \sum\limits_T x^{\Int(T)} y^{\Ext(T)},
$$
where the sum is over all spanning trees of $G$, and $\Int(T)$, resp. $\Ext(T)$, is the number of internally, resp. externally, active edges for the order $\prec_T$.
\end{definition}

\begin{remark}\label{rem:Tutte_descriptive}
In fact, Courtiel in \cite{Cour} introduces a more general notion of Tutte-descriptive activity. Our notion above corresponds to what he calls \emph{tree-compatible} order maps.
\end{remark}

\begin{theorem}\label{thm:Tutte_descriptive}
Let $G = G_{\pi}$ be a permutation graph, with sink $s$. Then the mapping $T \mapsto \prec_T$, where $\prec_T$ is the order defined by Algorithm~\ref{algo:<_T}, is Tutte-descriptive.
\end{theorem}

\begin{proof}
This follows from~\cite[Theorem 5.3]{Cour} in analogous fashion to the proof of~\cite[Proposition 7.9]{Cour}, with the slight adjustments necessary to take into account that Algorithm~\ref{algo:<_T} provides an order map based on a breadth-first, rather than depth-first, search.
\end{proof}

Combining Theorems~\ref{thm:ext_activity_level} and \ref{thm:Tutte_descriptive} gives a new combinatorial proof of the link between the level polynomial and the Tutte polynomial in Proposition~\ref{pro:Tutte_level} in the case of permutation graphs.

\section{Minimal recurrent configurations and complete non-ambiguous binary trees}\label{sec:minrec_CNAB}

\subsection{Minimal recurrent configurations}\label{sec:minrec}

Given a graph $G$ and a distinguished vertex $s$ of $G$, a configuration $c\in\Config{G}$ is minimal recurrent if it is recurrent and $\level{c}=0$. We denote by $\MinRec{s}{G}$ the set of minimal recurrent configurations for the ASM on $G$. We show that on permutation graphs, minimal recurrent configurations are uniquely determined by their canonical toppling.

\begin{definition}\label{def:comp_partition_permutation}
Given a permutation $\pi \in \clsn$ and a distinguished vertex $s \in [n]$, we say that an ordered set partition $P=P_0,\ldots,P_k$ of $[n]$ is \emph{$(\pi,s)$-compatible} if it satisfies the following three conditions:
\begin{enumerate}
\item\label{defcpp1} $P_0 = \{s\}$.
\item\label{defcpp2} For any $j \geq 0$, the elements of $P_j$ appear in increasing order in $\pi$ (that is, there is no inversion in $\pi$ between two elements of $P_j$).
\item\label{defcpp3} For any $j \geq 1$ and $i \in P_j$, there exists $i' \in P_{j-1}$ such that $(i,i')$ or $(i',i)$ is an inversion of $\pi$.
\end{enumerate}
\end{definition}

\begin{example}\label{ex:comp_partition_permutation}
Let $\pi = 25341$ and $s=3$. We wish to compute the set of $(\pi,s)$-compatible ordered partitions of $[5]$. We always have $P_0 = \{3\}$. From Condition~\eqref{defcpp3}, $P_1$ must be formed of elements $i$ such that $(i,3)$ or $(3,i)$ is an inversion of $\pi$. There are two such elements: $5$ and $1$. However, $(5,1)$ is an inversion of $\pi$, so $P_1$ cannot contain both these elements by Condition~\eqref{defcpp2}. Thus we must have $P_1 = \{1\}$ or $P_1 = \{5\}$. We now remark that the only element which forms an inversion with $2$ is $1$, so that, by Condition~\eqref{defcpp3}, $2$ must be in the part immediately after that containing $1$. Moreover, the part containing $2$ must either contain another element, or be the final part of $P$. 

Suppose that $P_1=\{1\}$. By the preceding argument, $P_2$ must contain $2$ and at least one other element which forms an inversion with $1$. There are two remaining elements which do this: $4$ and $5$. Since $(5,4)$ is an inversion, $P_2$ cannot contain both of these, so we must have $P_2 = \{2,4\}$ and $P_3 = \{5\}$, or $P_2 = \{2,5\}$ and $P_3 = \{4\}$. Suppose now that $P_1=\{5\}$. By similar arguments, we must have $P_2 = \{1\}$ or $P_2 = \{4\}$. Using the argument from the previous paragraph, if $P_2 = \{1\}$, then we must have $P_3=\{2,4\}$, and if $P_2 = \{4\}$, then we must have $P_3=\{1\}$ and $P_4=\{2\}$. Finally, we see that there are four $(\pi,s)$-compatible ordered partitions, which we write as blocks separated by dashes, for clarity: 
\def\bb{\raise1pt\hbox{--}}
\def\bb{\raise.6ex\hbox{\rule{.6em}{.1ex}}}
$$
3\bb1\bb24\bb5,\;\;\;\; 3\bb1\bb25\bb4,\;\;\;\; 3\bb5\bb1\bb24,\;\;\;\;  3\bb5\bb4\bb1\bb2.
$$
\end{example}

\begin{prop}\label{pro:minrec_orderedpartitions}
Let $\pi \in \clsn$ and $s \in [n]$. The map $\ctp : c \mapsto \CanonTop(c)$ is a bijection from the set $\MinRec{s}{G_{\pi}}$ of minimal recurrent configurations on the permutation graph $G_{\pi}$ to the set of $(\pi,s)$-compatible ordered partitions of $[n]$. 
\end{prop}

\begin{proof}
Let $c \in \MinRec{s}{G_{\pi}}$, and define $P := \CanonTop(c) = P_0,\ldots,P_k$. By definition, $P$ is an ordered partition of $[n]$ and $P_0 = \{s\}$. 
Let $T:=\ttc^{-1}(c)$ be the spanning tree of $G=G_{\pi}$ corresponding to $c$ via the inverse of the bijection in Theorem~\ref{thm:main_result}, and for any $i \in [n]$, let $\lambda_i(T),\mu_i(T),\nu_i(T)$ be defined as in Equations~\eqref{eq:def_lambda},~\eqref{eq:def_mu},~\eqref{eq:def_nu} in Section~\ref{sec:main_thm}.
By Theorem~\ref{thm:main_result}, we have $\CanonTop(c) = T^{(1)},\ldots,T^{(k)}$, that is, $P_j = T^{(j)}$ for all $j \in [k]$.
Moreover, since $c$ is minimal, we have $\level{c} = 0$, which in particular implies $\mu_i(T) = 0$ for all $i \in [n]$. Thus for any $j \in [k]$ there are no edges in $G$ between any two vertices of $P_j$. This implies that the elements of $P_j$ appear in increasing order in $\pi$, so Condition~\eqref{defcpp2} of Definition~\ref{def:comp_partition_permutation} is satisfied.

Now let $j \geq 2$ and $i \in T^{(j)}=P_j$. Let $i' = p(i)$ be the parent of $i$ in $T$, so that $i' \in T^{(j-1)} = P_{j-1}$. Since $(i',i)$ is an edge of $T$ it is also an edge of $G$, which means that $(i',i)$ or $(i,i')$ is an inversion of $\pi$, as desired. We have thus shown that if $c \in \MinRec{s}{G_{\pi}}$, then $\CanonTop(c)$ is a $(\pi,s)$-compatible ordered partition of $[n]$. This shows that the map of Proposition~\ref{pro:minrec_orderedpartitions} is well defined.

To show that it is a bijection, we define its inverse. Suppose that $P = P_0,\ldots,P_k$ is a $(\pi,s)$-compatible ordered partition of $[n]$. We first construct a spanning tree $T=T(P)$ of $G$. The tree $T$ will be rooted at $s$ so that for all $j \in [k]$ we have $T^{(j)} = P_j$. To define $T$ it is thus sufficient to define the parent map $p$. For $j \geq 1$, and $i \in P_j$, we define 
\begin{equation}\label{eq:def_parent}
p(i) := \min \left( N^G(i) \cap P_{j-1} \right).
\end{equation}
By Condition~\eqref{defcpp3} of Definition~\ref{def:comp_partition_permutation}, $p(i)$ is well defined, that is, $N^G(i) \cap P_{j-1} \neq \emptyset$.

We now define $c=\phi_{PC} (P) := \ttc(T(P))$, where $\ttc$ is the bijection of Theorem~\ref{thm:main_result}, and $T(P)$ is the tree defined above. We show that $c$ is minimal. Let $i \in [n]$. By Condition~\eqref{defcpp2} of Definition~\ref{def:comp_partition_permutation}, we have $\mu_i(T) = 0$ since there are no edges in $G$ between any two vertices of $T^{\left( h(i) \right)} = P_{h(i)}$. Moreover, Equation~\eqref{eq:def_parent} implies that $\nu_i(T)=0$. Since these are true for any $i \in [n]$ it follows from Theorem~\ref{thm:main_result} that $\level{c}=0$, as desired.

Finally, it is straightforward to show that the maps $\phi_{PC}$ and $\ctp$ are inverses of each other, which completes the proof.
\end{proof}

\subsection{Complete non-ambiguous binary trees}\label{sec:CNAB}

Non-ambiguous binary trees were introduced and studied in~\cite{ABBS} as a special case of the tree-like tableaux from~\cite{ABN}.

\begin{definition}\label{def:NAB}
A non-ambiguous binary tree (\nab) is a filling of a rectangular Ferrers diagram $F$ where every cell is either empty or dotted such that:
\begin{enumerate}
\item\label{nab:1} Every row and every column has a dotted cell.
\item\label{nab:2} Except for the top left cell, every dotted cell has either a dotted 
cell above it in its column or to its left in its row, but not both.
\end{enumerate}
 The dot in the top left cell (implied by \eqref{nab:1} and \eqref{nab:2}) is called the \emph{{root dot}}, or simply the \emph{root}.
\end{definition}

Through the remainder of this section, when we talk about a dot to the left/right of 
another dot we mean in the \emph{same row}, and similarly in the \emph{same column} for above/below.

The name \emph{non-ambiguous binary tree} comes from the fact that by drawing an edge between a dotted cell and the dotted cell immediately above it or to its left, for all dotted cells, one creates a binary tree, embedded in the grid $\mathbb{Z}_2$. Regarding the dot in the top left cell as a root of the tree, Condition~\ref{nab:2} of Definition~\ref{def:NAB} ensures that every other dot has a unique parent.  

A \nab\ is \emph{complete} if the associated binary tree is complete, that is, every dotted cell has either a dotted cell below it and to its right, or neither of these, and we refer to such complete \nabs\ as \cnabs. Figure~\ref{fig:ex_NAB} shows two examples of \nabs, with the edges of the associated binary tree drawn in. The left-hand one is complete, while the right-hand one is not.

\begin{figure}[h]
  \centering
  \begin{tikzpicture}[scale=0.35]
  \newcommand\dott[2]{\draw [fill=black] (#1,#2) circle [radius=0.25];}
  \begin{scope}[shift={(-1,-4)}]
  \draw[step=2cm,thick] (0,0) grid (10,10);
  \end{scope}
  \begin{scope}[shift={(-2,-7)}]
    \draw (2,6)--(2,12);
    \draw (2,12)--(10,12);
    \draw (2,10)--(8,10);
    \draw (4,8)--(4,12);
    \draw (6,4)--(6,12);

    \dott{2}{12}
    \dott{4}{12}
    \dott{6}{12}
    \dott{10}{12}

    \dott{2}{10}
    \dott{8}{10}

    \dott{4}{8}

    \dott{6}{4}


    \dott{2}{6}
  \end{scope}

    \begin{scope}[shift={(13,-3)}]
    \draw[step=2cm,thick] (0,0) grid (10,8);
    \draw (3,1)--(3,5)--(7,5)--(1,5)--(1,7)--(9,7)--(5,7)--(5,3);
    \draw [fill=black] (3,1) circle [radius=0.25];
    \draw [red,fill=red] (1,5) circle [radius=0.25];
    \draw [fill=black] (3,5) circle [radius=0.25];
        \draw [fill=black] (1,7) circle [radius=0.25];    
    \draw [fill=black] (5,3) circle [radius=0.25];
    \draw [fill=black] (5,7) circle [radius=0.25];
    \draw [fill=black] (7,5) circle [radius=0.25];
    \draw [fill=black] (9,7) circle [radius=0.25];
    \end{scope}
  \end{tikzpicture}
  
  \caption{Two examples of \nabs. The left-hand one is complete, and thus a \cnab, while the right-hand one is not, since the red vertex in column 1, row 3, has only one child. \label{fig:ex_NAB}}
\end{figure}
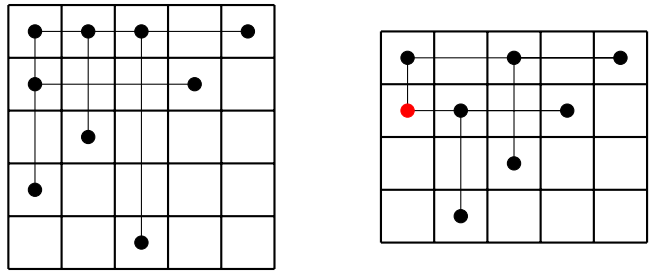

\begin{lemma}
A \nab\ on a Ferrers diagram $F$ has exactly $n$ dots, where $n$ is one less 
than the semi-perimeter of $F$.  If a  \nab\ is complete then $F$ has the same number of rows as columns.
\end{lemma}

\begin{proof}
Each non-root dot has a dot above it or to its left, but not both.  If such a dot has no dot above it, move it to the top row.  Otherwise it has no dot to its left, in which case move it to the leftmost column. This moves every dot either to the top row or leftmost column. (We regard dots in the top row and leftmost column as being moved, although they stay put.) Every column has a dot that will be moved up, namely the column's topmost dot, and every row's leftmost dot moves to the leftmost column.  This process will therefore leave dots in the entire top row and leftmost column, but nowhere else, which proves the first part.

Given a complete NAB, we trace the above process of moving dots to the top row or leftmost column.  For each non-root dot that gets moved to the top row the dot to its left (its parent) has a dot below it, which therefore gets moved into the leftmost column, and conversely.  Thus, there must be as many rows as columns.
\end{proof}


\subsection{Complete non-ambiguous binary multitrees}\label{sec:CNABM}

Given a permutation $\pi=\pi_1\pi_2\ldots\pi_n$ let $\tpi$ be the $n\times n$ grid with dots in cells $(\pi_i,i)$, where cell $(i,j)$ is the cell in row $i$ and column $j$,
the northwest corner cell being $(1,1)$.

\begin{definition}\label{def:cmNAB}
A \emph{complete multirooted non-ambiguous binary tree (CMNAB)} is obtained from $\tpi$, for a permutation $\pi$, by adding a further $n-1$
 dots with the following conditions:
\begin{enumerate}
\item\label{mnab-cond1} Every added dot has a dotted cell below it in its column and to its right in its row.
\item\label{mnab-cond2} The graph obtained as in the case of a \cnab\ is a tree.
\end{enumerate}
The added dots are called \emph{internal dots}. The set of \cmnabs\ arising from~$\tpi$ is denoted~$\mpi$, and the tree obtained from $M\in\mpi$ is denoted $T(M)$.
\end{definition}

A \cmnab\ gives a multirooted tree where the roots are the dots 
with no dots to the left or above; see Figure~\ref{fig:map_zeta}.  We will show that \cnabs\ are precisely the \cmnabs\ with a single root. 

\begin{lemma}\label{lem:cell2edge}
A cell $(i,j)$ in $\tpi$ has a dot to the right and below if and only if 
there is an edge between $i$ and $\pi_j$ in $G_\pi$.
\begin{proof}
If a cell $(i,j)$ has dots both to the right and below then there is a leaf dot below it in row $r>i$, so $\pi_j=r$, and to its right in column $c>j$, so $\pi_c=i<r$. Since $j<c$ and $\pi_j>\pi_c=i$, $\pi_j$ and $i$ form an inversion in $\pi$, so $(\pi_j,i)$ is an edge in $G_\pi$. 

Conversely,  an edge in $G_\pi$ corresponds to an inversion in $\pi$, which in turn corresponds to a pair of leaf dots in $\pi$, the leftmost of which is lower than the other, and thus there is a cell above the leftmost one and to the left of the other.
 \end{proof}
\end{lemma}

By Lemma~\ref{lem:cell2edge} every cell in $\tpi$ with an internal dot corresponds to an 
edge in $G_\pi$.  So we can map the elements of $\mpi$ to subgraphs of $G_\pi$, by the map $\zeta$ which maps $M\in\mpi$ to the subgraph $\zeta(M)$ of $G_\pi$ with edge set
$$
E(\zeta(M))=\{(i,\pi_j) : (i,j)\text{ contains an internal dot in }M\}.
$$
Note that the non-internal dots in $M$ are the leaves of $T(M)$ and they correspond precisely to the pairs $(i,j)$
 where $\pi_j=i$.

The following lemma is straightforward to prove.

\begin{lemma}\label{lem:paths}
Let $S=\zeta(M)$ for a \cmnab\ $M\in \mpi$, so $S$ is a subgraph of $G_\pi$. The sequence 
$$
v_1,e_1,v_2,e_2,\ldots,e_{k-1},v_k,
$$
alternating between vertices and edges in $S$, is a path in $S$ if and only if 
$$
\ell_1,i_1,\ell_2,i_2,\ldots,\ell_{k-1},i_k
$$
is an alternating sequence of leaf and internal dots in $M$, with every pair of consecutive dots in the same row or column, where $\ell_t$ and $i_t$ are the dots corresponding to the vertex $v_t$ and edge $e_t$, respectively. In particular, $T(M)$ being connected is equivalent to $S$ being connected.  Moreover, adding an edge to $S$ corresponds to closing such a sequence through $M$ to a cycle.
\end{lemma}

\begin{figure}[h]
\centering
\begin{tikzpicture}[scale=0.35]
  \newcommand\dott[2]{\draw [fill=black] (#1,#2) circle [radius=0.25];}
  \draw[step=2cm,thick] (0,0) grid (12,12);
  \begin{scope}[shift={(-1,-1)}]
    \draw (2,6)--(2,10);
    \draw (4,12)--(10,12);
    \draw (2,10)--(8,10);
    \draw (4,2)--(4,12);
    \draw (6,4)--(6,12);
    \draw (4,8)--(12,8);

    \dott{4}{12}
    \dott{6}{12}
    \dott{10}{12}

    \dott{2}{10}
    \dott{4}{10}
    \dott{8}{10}

    \dott{4}{8}
    \dott{12}{8}

    \dott{6}{4}

    \dott{4}{2}

    \dott{2}{6}

    \node at(2,14.3){$1$};
    \node at(4,14.3){$2$};
    \node at(6,14.3){$3$};
    \node at(8,14.3){$4$};
    \node at(10,14.3){$5$};
    \node at(12,14.3){$6$};

    \node at(-0.3,12){$1$};
    \node at(-0.3,10){$2$};
    \node at(-0.3,8){$3$};
    \node at(-0.3,6){$4$};
    \node at(-0.3,4){$5$};
    \node at(-0.3,2){$6$};

  \end{scope}

  \begin{scope}[shift={(18,0)}]
    \newcommand\nod[3]{\draw[fill=white](#1,#2)circle[radius=.75];\node at(#1,#2){$#3$};}
    \begin{scope}[shift={(-1,-1)}]
      \draw[thin] (4,12)--(10,12)--(12,7)--(10,2)--(4,2)--(2,7)--(4,12);
      \draw[thin] (4,2)--(4,12)--(12,7)--(2,7)--(10,12)--(10,2)--(4,2);

      \draw[red, ultra thick] (2,7)--(10,12)--(4,12)--(12,7)--(10,2);
      \draw[red, ultra thick] (4,2)--(4,12);

      \nod{4}{12}{6}
      \nod{10}{12}{1}
      \nod{2}{7}{5}
      \nod{12}{7}{2}
      \nod{4}{2}{3}
      \nod{10}{2}{4}
    \end{scope}
  \end{scope}
\end{tikzpicture}

\caption{An element $M\in\mathcal{M}_{465213}$, and the graph $G_{465213}$, with the spanning tree corresponding to $M$ marked with thick red lines. Moving the dot at $(2,2)$ to $(1,1)$, thus creating a \cnab, would correspond to replacing the edge $(2,6)$ by $(1,4)$ in the spanning tree, whereas moving the dot at $(2,1)$ to $(3,1)$ corresponds to replacing the edge $(2,4)$ with $(3,4)$.\label{fig:map_zeta}}
\end{figure}
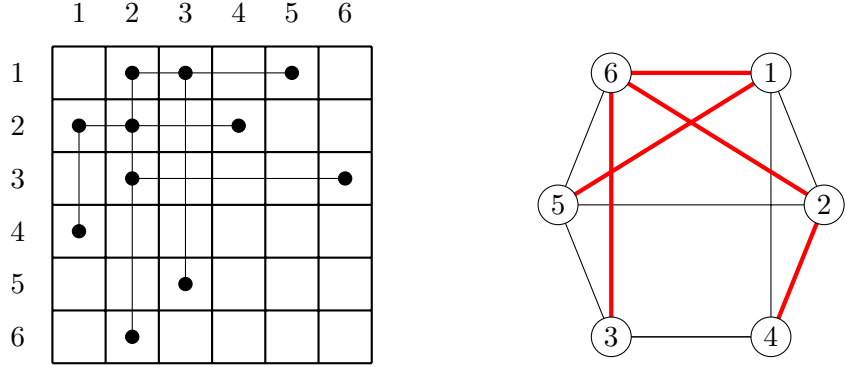

\begin{prop}
The map $\zeta$ is a bijection from $\mpi$ to the spanning trees of $G_\pi$.
\end{prop}
\begin{proof}
  Let $\pi$ be an $n$-permutation.  By Lemma~\ref{lem:cell2edge} and the fact that every element of $\mpi$ has $n-1$ internal dots, $\zeta$ is a map from $\mpi$ to the set of subgraphs of $G_\pi$ with $n-1$ edges.  To show that those edges form a spanning tree it therefore suffices to show that they form a connected graph, which follows from Lemma~\ref{lem:paths}.

Conversely, if $S$ is a spanning tree of $G_\pi$, place a dot in cell $(i,j)$ of $\tpi$ for each edge $(i,\pi_j)$ of $S$, where $i<\pi_j$.  By Lemma~\ref{lem:cell2edge}, this places $n-1$ internal dots in $\tpi$ and each of those dots contributes two edges to the graph in $\tpi$, connecting to a dot below and to the right, a total of $2n-2$ edges, in a graph with $2n-1$ vertices.  To show that this graph is a tree it again suffices to show that it is connected, which again follows from Lemma~\ref{lem:paths}.
\end{proof}

If a \cmnab\ has a unique root then the definition is equivalent to that of a \cnab, as we will now show.

\begin{lemma}\label{lem:1root}
A \cmnab\ $M$ has a dot with a dot to the left and above if and only if it has more than one root.
\end{lemma}
\begin{proof}
Suppose $M$ has a dot $d$ with a dot $a$ above and a dot $\ell$ to the left.  Tracing a zig-zag path from $d$ through $a$ to the topmost dot in their column, then to the leftmost dot in that row, and so on, we must end up at a dot with no dot above or to the left, which is a root dot~$d_1$. Tracing analogously through $\ell$ we will end up at a root dot $d_2$. These two root dots must be distinct, for otherwise we would have traced out a cycle in the tree $T(M)$.

Suppose then that $M$ has (at least) two distinct root dots $\ell$ and $h$, which then must be in different rows, say $h$ in a higher row.  The unique path from $\ell$ to $h$
in the tree $T(M)$ must contain an up-step, but start with a right or a down step. Consider the maximal sequence of right and down steps in the beginning of that path.  If the last step in that sequence was a right step the next step must be up, if it was a down step the next step must be left.  In either case we have found a dot with a dot to the left and above.
\end{proof}

\begin{prop}\label{prop:cnabs-cmnabs}
The \cmnabs\ with a single root are precisely the \cnabs.
\end{prop}

\begin{proof}
Every row in a \cnab\ has a unique leaf dot, which is also the unique leaf dot in its column, and this accounts for $n$ dots. The remaining $n-1$ dots are internal and satisfy Condition~\eqref{mnab-cond1} in Definition~\ref{def:cmNAB}, and \cnabs\ satisfy Condition~\eqref{mnab-cond2} in Definition~\ref{def:cmNAB}, so every \cnab\ is a \cmnab.  Conversely, if a \cmnab\ $M$ has a single root, then by Lemma~\ref{lem:1root} no dot has a dot to the left and above, and so $M$ satisfies Condition~\eqref{nab:2} in Definition~\ref{def:NAB} (and Condition~\eqref{nab:1} by definition) and thus is a \cnab.
\end{proof}

We can use the map $\zeta$ to map the \cnabs\ on $\tpi$ to the minimal recurrent configurations on~$G_\pi$.  For the following lemma we order the edges of $G_\pi$ reverse lexicographically by coordinates of the corresponding cells in $\tpi$ (see Definition~\ref{def:activity} of external activity).

\begin{prop}\label{prop:root-ext}
Let $M$ be a \cmnab. There is a unique root in $M$ if and only if $\zeta(M)$ has no external activity.
\begin{proof}
If $M$ has more than one root, then Lemma~\ref{lem:1root} implies there exists a dot $d$ in $M$ with a dot $a$ above and dot $\ell$ to the left.  Let $c$ be the cell that completes the rectangle of $a,\ell$ and~$d$.  Then $c$ corresponds to an edge external to the spanning tree $S$ and adding it to $S$ creates a cycle with edges corresponding to $a,\ell,c$ and $d$, by Lemma~\ref{lem:paths}. Since the edge corresponding to $c$ is ordered last of these edges it is externally active.
 
If $e$ is an externally active edge then adding it creates a cycle in the spanning tree $S$, which corresponds to a cycle of internal dots in $M$. Such a cycle must contain a dot with a dot to the left and a dot above. However, as $e$ is externally active it must be ordered last in its cycle in~$S$ so the corresponding dot in $M$ is weakly northwest of all other dots in the cycle. Thus the addition of the dot corresponding to $e$ cannot cause one of the pre-existing dots to have a dot to the left and above.  Therefore, such a dot must already exist in the cycle, so $M$ has a cell with a dot to the left and a dot above, which implies $M$ is multirooted, by Lemma~\ref{lem:1root}.
\end{proof}
\end{prop}

Note that in this case, unlike in Section~\ref{sec:Tutte}, the order of the edges of 
$G=G_{\pi}$ is fixed \textit{a priori} (it does not depend on the tree $T$). It is known (see~\cite{Tutte}) that in this case we have
$$
\tut_G(x,y) = \sum\limits_T x^{\Int(T)} y^{\Ext(T)},
$$
where the sum is over all spanning trees of $G$, and thus by Proposition~\ref{pro:Tutte_level} the spanning trees with no external activity are in bijection with the minimal recurrent configurations. Therefore, Proposition~\ref{prop:cnabs-cmnabs} and Proposition~\ref{prop:root-ext} imply the following.
\begin{corollary}\label{cor:1root-minrec}
The elements of $\mpi$ with a single root are in bijection with 
the minimal recurrent configurations of $G_\pi$.
\end{corollary}

\begin{problem}
Find a nice bijective proof of Corollary~\ref{cor:1root-minrec}.
\end{problem}

\begin{remark}
In~\cite{DGGS}, the authors provided a new interpretation of the sequence $A002190 = 1,1,4,33,456,9460,\ldots$
 in~\cite{oeis} counting complete non-ambiguous binary trees, in terms of fully tiered trees with weight $0$. Section~\ref{sec:minrec} and Corollary~\ref{cor:1root-minrec} provide another two combinatorial interpretations to this sequence:
\begin{itemize}
\item as the sum $\sum\limits_{ \pi \in \bar{\cls}_n} \vert \MinRec{s}{G_{\pi}}\vert$, where $\bar{\cls}_n$ is the set of indecomposable permutations of length $n$.
\item as the number of pairs $(\pi,P)$ where $\pi \in \bar{\cls}_n$ and $P$ is a $(\pi,s)$-compatible ordered partition of $n$.
\end{itemize}
\end{remark}


\section{Specialisations}\label{sec:specialisations}

\subsection{The Ferrers case}\label{sec:Ferrers}

In this section, we are interested in the case where the permutation~$\pi$ has a single descent, in which case the permutation graph $G_{\pi}$ is a Ferrers graph. In this case the spanning trees of the permutation graph are the intransitive trees introduced by Postnikov~\cite{Post}. As such, we recover results from~\cite[Section 5.3]{DSSS}.

A \emph{Ferrers graph} (see~\cite{ew-ferrers}) is a bipartite graph whose vertices of 
each part are labeled $t_1,t_2,\ldots,t_k$ and $b_1,b_2,\ldots,b_m$, respectively, satisfying the following conditions:
\begin{enumerate}
\item\label{fg1} If $(t_i,b_j)$ is an edge with $r\le i$ and $s\le j$, then $(t_r,b_s)$ is also an edge.
\item\label{fg2} Both $(t_1,b_m)$ and  $(t_k,b_1)$ are edges.
\end{enumerate}
As illustrated in the example in Figure~\ref{fig:Ferrers_permgraph}, we think of the 
vertices $t_i$ as ``top'' vertices, and the~$b_i$ as ``bottom'' vertices. Note that when read from left to right the labels on the top vertices are increasing  but decreasing for the bottom vertices. Thus, condition~\eqref{fg1} above says that a top vertex must have edges to all vertices that any vertex to its right does, and likewise for bottom vertices.

\begin{figure}[h]
  \centering
  \begin{tikzpicture}[scale=0.35]
\begin{scope}[shift={(-2,0)}]
    \draw[step=2cm,thick] (0,2) grid (8,8);
    \draw[step=2cm,thick] (6,6) grid (10,8);
    \draw[step=2cm,thick] (0,0) grid (2,2);
    
    \node at (10.5,7) {$1$};
    \node at (9,5.5) {$2$};
    \node at (8.5,5) {$3$};
    \node at (8.5,3) {$4$};
    \node at (7,1.5) {$5$};
    \node at (5.1,1.5) {$6$};
    \node at (3.1,1.5) {$7$};
    \node at (2.5,.7) {$8$};
    \node at (1,-.5) {$9$};
    
    \node at (-0.5,7) {$t_1$};
    \node at (-0.5,5) {$t_2$};
    \node at (-0.5,3) {$t_3$};
    \node at (-0.5,1) {$t_4$};
    \node at (1,8.7) {$b_1$};
    \node at (3,8.7) {$b_2$};
    \node at (5,8.7) {$b_3$};
    \node at (7,8.7) {$b_4$};
    \node at (9,8.7) {$b_5$};
\end{scope}  

\begin{scope}[shift={(13,2)}]
	\draw (1.5,4)--(0,0)--(1.5,4)--(3,0)--(1.5,4)--(6,0)--(1.5,4)--(9,0)--(1.5,4)--(12,0);
	\draw (4.5,4)--(0,0)--(4.5,4)--(3,0)--(4.5,4)--(6,0)--(4.5,4)--(9,0);
	\draw (7.5,4)--(0,0)--(7.5,4)--(3,0)--(7.5,4)--(6,0)--(7.5,4)--(9,0);
	\draw (10.5,4)--(0,0);
	\draw [fill=white] (1.5,4) circle [radius=0.8];
	\draw [fill=white] (4.5,4) circle [radius=0.8];
	\draw [fill=white] (7.5,4) circle [radius=0.8];
	\draw [fill=white] (10.5,4) circle [radius=0.8];
	\draw [fill=white] (0,0) circle [radius=0.8];
	\draw [fill=white] (3,0) circle [radius=0.8];
	\draw [fill=white] (6,0) circle [radius=0.8];
	\draw [fill=white] (9,0) circle [radius=0.8];
	\draw [fill=white] (12,0) circle [radius=0.8];
	\node at (1.5,4) {$1$};
	\node at (4.5,4) {$3$};
	\node at (7.5,4) {$4$};
	\node at (10.5,4) {$8$};
	\node at (0,0) {$9$};
	\node at (3,0) {$7$};
	\node at (6,0) {$6$};
	\node at (9,0) {$5$};
	\node at (12,0) {$2$};
\end{scope}   
\end{tikzpicture}
  \caption{Example of a Ferrers diagram, the labeling of its South-East border, and the corresponding labeled Ferrers graph, which is exactly the permutation graph $G_{256791348}$.\label{fig:Ferrers_permgraph}}
\end{figure}
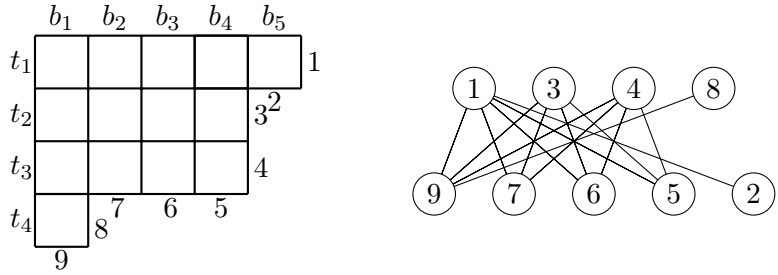

Given a Ferrers diagram with rows labeled from top to bottom with $t_1,t_2,\ldots,t_k$ and
columns labeled with $b_1,b_2,\ldots,b_m$ from left to right,
there is a unique Ferrers graph whose vertices are labeled with the $t_i$ and $b_j$ and where $(t_i,b_j)$ is an edge if and only if the diagram has a cell in row $t_i$ and column $b_j$.  This correspondence is clearly one-to-one.

Given a permutation $\pi \in \clsn$, we say that the pair $\pi_i,\pi_{i+1}$ is a \emph{descent} of $\pi$ if $\pi_i > \pi_{i+1}$.

\begin{prop}\label{pro:Ferrers_single_descent}
Let $G$ be a graph on $n$ vertices. Then $G$ is a Ferrers graph if and only if there exists an indecomposable permutation $\pi$ with exactly one descent such that $G \simeq G_{\pi}$.
\end{prop}

\begin{proof}
Suppose that $\pi$ is indecomposable and has a single descent. Then we can decompose~$\pi$ in a unique way as $\pi= \pi_1 \pi_2$, where $\pi_1$ and $\pi_2$ are increasing subsequences of $[n]$ and the last letter of $\pi_1$ is strictly greater than the first letter of $\pi_2$. Since $\pi$ is indecomposable, this implies that the the last letter of $\pi_1$ is $n$, and the first letter of $\pi_2$ is $1$. Now we let $F=F(\pi)$ be the Ferrers diagram 
defined as follows. Label the edges on the South-East border of $F$ from North-East to South-West in the order $1,2,\ldots,n$, and let the step labeled $k$ be a South step if $k \in \pi_2$ and a West step if $k \in \pi_1$. This defines a Ferrers diagram since $1 \in \pi_2$. Then the edges of the corresponding Ferrers graph $G(F)$ are the pairs $(i,j)$ where $i$ is a column label (a West step), $j$ a row label (a South step), and $i>j$, that is, exactly the inversions of~$\pi$. Thus $G_{\pi} \simeq G(F)$. 

For the converse, suppose $G=G(F)$ is the Ferrers graph corresponding to the Ferrers diagram $F$, whose South-East border is labeled as before. Let $\pi_1$ and $\pi_2$ be the words formed of West steps and East steps, respectively, of that border, each in increasing order. Then $\pi = \pi_1 \pi_2$ is a permutation of length $n$ with exactly one descent, and as above, we have $G \simeq G_{\pi}$, as desired (that $\pi$ is indecomposable follows from the fact that a Ferrers graph is connected).
\end{proof}

Figure~\ref{fig:Ferrers_permgraph} illustrates the construction in the proof of Proposition~\ref{pro:Ferrers_single_descent}. Thus, Ferrers graphs can be viewed as permutation graphs where the permutation has a single descent. 

\begin{remark}
It is possible for a permutation with more than one descent to yield a Ferrers graph. For instance, the graph corresponding to the permutation $3142$ is isomorphic to $P_4$, the path graph on $4$ vertices, which is a Ferrers graph. Indeed, $P_4$ is the Ferrers graph corresponding to the diagram whose row lengths are $(2,1)$, or equivalently, it is isomorphic to the permutation graph $G_{2413}$.  
\end{remark}

We revisit Theorem~\ref{thm:main_result} in the context of Ferrers graphs. Let $\pi \in \clsn$ be an indecomposable permutation with a single descent, and $G=G_{\pi}$ the corresponding Ferrers graph. We decompose $\pi$ into $\pi_1 \pi_2$ as before, where $\pi_1$ and $\pi_2$ are two increasing sequences such that the last letter of $\pi_1$ is $n$ and the first letter of $\pi_2$ is $1$. We write $A_1$ and $A_2$ for the unordered set of labels appearing in $\pi_1$ and $\pi_2$, respectively. For $j\in \{1,2\}$, we set $\bar{j} := 3-j$, so that if $j=1$, $\bar{j} = 2$ and vice versa.

\begin{lemma}\label{lemma:tree_levels}
Let $s \in [n] $ be a distinguished vertex of $G=G_\pi$ where $\pi$ has a single descent, and let $j \in \{1,2\}$ be such that $s \in A_j$. Let $T$ be a spanning tree of $G$, rooted at $s$. Then for any $k \geq 0$, we have:
\begin{itemize}
\item $T^{(2k)} \subseteq A_j$.
\item $T^{(2k+1)} \subseteq A_{\bar{j}}$.
\end{itemize}
\end{lemma}

\begin{proof}
Any edge $e$ of $G$ is a pair $e=(x,y) \in A_1 \times A_2$. Thus if $e = (x,y) \in T^{(k)} \times T^{(k+1)}$ is an edge of $T$, we have that if $x \in A_j$ then $y \in A_{\bar{j}}$ and vice versa. Since $T^{(0)} = \{s\} \subseteq A_j$, the claim follows immediately by induction. 
\end{proof}

An immediate consequence of Lemma~\ref{lemma:tree_levels} is the following.

\begin{prop}\label{pro:mu=0}
Let $s \in [n] $ be a distinguished vertex of $G=G_\pi$ where $\pi$ has a single descent, and $T$ be a spanning tree of $G$, rooted at $s$. Then, for all $i\in[n]$, we have $\mu_i(T)=0$, where the $\mu_i(T)$ are defined as in Equation~\eqref{eq:def_mu} in Section~\ref{sec:main_thm}.
\end{prop}

\begin{proof}
Given $k \geq 0$, Lemma~\ref{lemma:tree_levels} implies that  $G$ has no edges between two elements of $T^{(k)}$.
\end{proof}

We now show that in this case, there is a one-to-one correspondence between spanning 
trees of permutation graphs $G_\pi$ where $\pi$ has a single descent, and the intransitive trees first introduced by Postnikov~\cite{Post}. 
Let~$T$ be a labeled tree on the vertex set $[n]$. We say that $T$ is \emph{intransitive} if all its vertices are either local minima or local maxima.
Given a tree $T$, we denote by $\MIN(T)$ and $\MAX(T)$ its set of local minima and maxima, respectively. Thus~$T$ is an intransitive tree if and only if $\MIN(T), \MAX(T)$ forms a partition of $[n]$.

The following Proposition is essentially a re-writing of Proposition~\ref{pro:tieredtrees_permgraphs} in the case of permutations with a single descent, but we give a proof in the current context.

\begin{prop}
Let $\pi$ be an indecomposable permutation with a single descent. Write $\pi=\pi_1 \pi_2$ with $\pi_1$ and $\pi_2$ being, respectively, the increasing ordering of a set $A_1$ containing $n$ and of a set $A_2$ containing $1$. Let $T$ be a labeled tree on the vertex set $[n]$. Then $T$ is a spanning tree of $G_{\pi}$ if and only if $T$ is an intransitive tree with $\MAX(T)=A_1$ and $\MIN(T) = A_2$.
\end{prop}

\begin{proof}
Suppose that $T$ is a spanning tree of $G=G_{\pi}$, and let $i \in A_1$. By construction, if $(i,j)$ is an edge of $G$, then $j \in A_2$ and $i>j$. In particular, all neighbors of $i$ in $T$ have labels strictly less than $i$, so that $A_1 \subseteq \MAX(T)$. Similarly, we have $A_2 \subseteq \MAX(T)$, and since $A_1,A_2$ forms a partition of $[n]$ this implies that $T$ is an intransitive tree with $\MAX(T)=A_1$ and $\MIN(T) = A_2$.
The converse follows from the fact that if $T$ is an intransitive tree, all its edges connect a local maximum $i$ to a local minimum $j$ with $i>j$.
\end{proof}

We now restate Theorem~\ref{thm:main_result} in this specialized context.
Let $\pi$ be an indecomposable permutation with a single descent. Write $\pi=\pi_1 \pi_2$ with $\pi_1$ and $\pi_2$ being, respectively, the increasing ordering of a set $A_1$ containing $n$ and of a set $A_2$ containing $1$. Let $G=G_{\pi}$ be the corresponding permutation graph (which is a Ferrers graph by Proposition~\ref{pro:Ferrers_single_descent}). Let $s \in [n]$ be a distinguished vertex of $G$, and $T$ a labeled tree on $[n]$, rooted at $s$. For $i \in [n]$, we define:
\begin{eqnarray}
\tilde{\lambda}_i(T) & := & \begin{cases}
\left\vert \{ j \in T^{\left(>h(i)\right)} \cap A_2 : \, j<i \} \right\vert, \quad \text{if } i \in A_1, \\
\left\vert \{ j \in T^{\left(>h(i)\right)} \cap A_1 : \, j>i \} \right\vert, \quad \text{if } i \in A_2.
\end{cases}
\label{eq:def_lambda_fer}\\
\tilde{\nu}_i(T) & := & 
\left\vert T^{\left(h(i) -1\right)} \cap (p(i),i) \right\vert
\label{eq:def_nu_fer},
\end{eqnarray}
where $p(i)$ is the parent of $i$ in the rooted tree $T$, and $(p(i),i)$ is the interval 
$[p(i)+1,i-1]$ if $p(i)<i$, and $(p(i),i)=[i+1,p(i)-1]$ if $i<p(i)$.

\begin{theorem}\label{thm:recconfig_intrantrees}
The map $T \mapsto c(T)$, with $c(T) \in \Config{G}$ defined by $c_i(T) := \tilde{\lambda}_i(T) + \tilde{\nu}_i(T)$, is a bijection from the set of intransitive trees $T$ such that $\MAX(T)=A_1$ and $\MIN(T)=A_2$, to the set $\Rec{s}{G}$ of recurrent configurations on $G$.

Moreover, we have $ \level{c(T)} = \sum\limits_{i=1}^n \left( \frac12 \tilde{\mu}_i(T) + \tilde{\nu}_i(T) \right) $, and $\CanonTop \left( c(T) \right) = T^{(0)},T^{(1)},\ldots$. 
That is, the canonical toppling of $c(T)$ is given by the breadth-first search of $T$.
\end{theorem}

In particular, we recover the bijection between the set of intransitive trees on $n$ vertices and the set of recurrent configurations on Ferrers graphs on $n$ vertices from~\cite{DSSS}. 

Note that the definition of $\tilde{\nu}$ in Equation~\eqref{eq:def_nu_fer} differs slightly from that of $\nu$ in Equation~\eqref{eq:def_nu} in Section~\ref{sec:main_thm} (the definition of $\tilde{\lambda}$ is the same as that of $\lambda$, though written slightly differently). This is due to the extra structure of intransitive trees, namely that every vertex is either a local minimum or a local maximum, which allows this simpler formula to be given. There is no additional difficulty in the proof, it merely requires a slight tweaking of the inverse map introduced in the proof of Theorem~\ref{thm:main_result}.

\subsection{Threshold graphs}\label{sec:thresholdgraphs}

Threshold graphs were introduced by Chv\'atal and Hammer~\cite{CH} and are defined as 
those graphs that can be constructed from a graph with one vertex by repeatedly adding an isolated vertex or a vertex that is connected to every already existing vertex.  It is easy to see that a threshold graph is the permutation graph of a permutation obtained from the permutation 1 by repeatedly appending or prepending a new largest letter.  One example of such a permutation is $86521347$; these are exactly the permutations that first 
decrease and then increase.  Note, however that a threshold graph may be disconnected and thus correspond to a decomposable permutation (which will have its largest letter last).

In~\cite{PYY} the authors present a general bijection between the parking functions of a graph and labeled spanning trees. In the case where $G$ is a threshold graph, this bijection maps the degree of the parking function to the number of inversions of the 
spanning tree (an inversion of a tree $T$ with vertex set $[n]$ is a pair $(i,j)$ such that $i>j$ and $j$ is an ancestor of $i$ in~$T$, relative to a designated root). Parking functions of a graph are essentially the same as recurrent configurations for the ASM, via a simple linear translation, with the degree of a parking function corresponding to the level of a recurrent configuration. Thus, the bijection in~\cite{PYY} can be viewed as a bijection between recurrent configurations on a threshold graph~$G$ and spanning trees of $G$, mapping the level statistic of the configuration to the number of inversions of the trees.

In Section~\ref{sec:Tutte}, we showed that our bijection in Theorem~\ref{thm:main_result} can be seen as a bijection between recurrent configurations on a permutation graph $G$ and spanning trees of $G$, mapping the level of the configuration to the external activity of the tree. 
It is known that the inversion and external activity statistics are equidistributed over labeled trees on $n$ vertices, and~\cite{Beis} provides a bijective proof of this fact. 
Since threshold graphs are a special case of permutation graphs, it follows that Theorem~\ref{thm:main_result} can be viewed as an extension of the work in~\cite{PYY}.

\bibliographystyle{abbrv}
\bibliography{ASM_Permgraphs}

\end{document}